\documentclass[reqno]{amsart}
\usepackage{color}
\usepackage{amssymb}
\usepackage[dvipsnames]{xcolor}
\usepackage{hyperref}
\usepackage{enumitem}
\usepackage{amsmath}

\newtheorem{theorem}{Theorem}[section]
\newtheorem{lemma}{Lemma}[section]
\newtheorem{proposition}{Proposition}[section]

\theoremstyle{definition}

\theoremstyle{remark}
\newtheorem{remark}[theorem]{Remark}

\numberwithin{equation}{section}

\newcommand{\R}{\mathbb{R}}
\newcommand{\C}{\mathbb{C}}
\newcommand{\Z}{\mathbb{Z}}
\newcommand{\N}{\mathbb{N}}

\newcommand{\begeq}{\begin{equation}}
\newcommand{\stopeq}{\end{equation}}
\newcommand{\ep}{\epsilon}

\newcommand{\ds}{\displaystyle}

\newcommand{\Om}{\Omega}

\newcommand{\ta}{\theta}

\allowdisplaybreaks

\begin{document}
\title[Infinitesimal Bendings]
{Infinitesimal Bendings for Classes of  Two Dimensional Surfaces.}

\author{B. de Lessa Victor}
\thanks{\small The first author was financed in part by the Coordenação de Aperfei\c coamento de Pessoal de N\'ivel Superior - Brasil (CAPES) - Finance Code 001 and in part by FAPESP (grant number 2021/03199-9) } 
\address{\small Departamento de Matemática, Instituto de Ciências Matemáticas e de \linebreak  Computação (ICMC), Universidade de São Paulo (USP), São Carlos (SP), Brazil.}
\email{brunodelessa@gmail.com}

\author{Abdelhamid Meziani}
\address{\small Department of Mathematics, Florida International University, Miami, FL, 33199, USA.}
\email{meziani@fiu.edu}



\begin{abstract}
Infinitesimal bendings for classes of two-dimensional surfaces in $\R^3$ are investigated.
The techniques used to construct the bending fields include reduction to solvability of
Bers-Vekua type equations and systems of differential equations with periodic coefficients.   
\end{abstract}
\maketitle

\section{Introduction}
This paper deals with infinitesimal bendings for classes of orientable surfaces.
We consider a smooth surface $S\subset\R^3$ given by a position vector $R$
over a region $\Om\subset\R^2$. Thus
\[
S=\{ R(s,t)\in\R^3;\ (s,t)\in\Om\},
\]
where $R\in C^\infty(\Om,\R^3)$. A one parameter deformation surface $S_\ep$ ($\ep\in\R$) given by the position
vector
\[
R_\ep(s,t)=R(s,t)+2\sum_{j=1}^m\ep^jU_j(s,t)\, ,
\]
with $U_j\in C^k(\Om,\R^3)$ ($\, k\in\Z^+$), is an
infinitesimal bending of $S$ of order $m\in\Z^+$ if the metrics of $S$ and $S_\ep$ coincide to order
$m$ as $\ep\to 0$. That is,
\[
dR_\ep^2(s,t)=dR^2(s,t)+ o(\ep^m)\quad\textrm{as}\quad \ep\to 0.
\]
The study of infinitesimal bendings of surfaces has a long and rich history and has many physical applications
(see \cite{ni}, \cite{po}, and \cite{ro}). For a complete overview we refer to the survey article by Sabitov \cite{s} and the extensive
references within.

The results of this paper are generalizations of those contained in \cite{m1}, \cite{m3} that deal with infinitesimal
bendings of surfaces with nonnegative curvature. For a surface with positive Gaussian curvature except at a finite number
of planar points, we use the (complex) vector field of asymptotic directions and an associated Bers-Vekua type equation
to construct non trivial infinitesimal bendings of any finite order (Theorem \ref{Prisoners}). For surfaces with nonnegative curvature given
as a graph of a homogeneous function: $R(s,t)=(s,t,z(s,t))$ with $z$ a homogeneous function,
we construct infinitesimal bendings of
higher orders through the solvability of associated systems of periodic differential equations, provided
that two numbers attached to the surface 
satisfy a number theoretic condition
(Theorems \ref{Homogeneous1} and \ref{Homogeneous2}).

In the final section, we consider a special class of surfaces defined as a graphs of
function $s^{m+2}\pm t^{n+2}$ (a model for surfaces defined as graphs of functions $f(s)+g(t)$).
We show (Theorem \ref{AnalyticBendings}) that the space of
real analytic infinitesimal bendings on the rectangle $|s|<\rho, |t| <\rho$ with $0<\rho\le\infty$,
is isomorphic to the space $\mathcal{A}(\rho)^4$,
where $\mathcal{A}(\rho)$ is the space of convergent power series of one variable with radius of convergence $\rho$.

\section{Definitions and Equations for Bending Fields}

Let $S$ be a $C^{\infty}$ surface in $\R^{3}$ over a domain $\Omega \subset \R^{2}$ given by
\begin{equation} \label{Schindler's List}
S = \left\{R(s,t) = \left(x(s,t), y(s,t), z(s,t) \right); \ (s, t) \in \Omega  \right\},
\end{equation}
and $S_\varepsilon$  a deformation of $S$, given by
\begin{equation} \label{Sunset Boulevard}
S_\varepsilon = \left\{R_{\varepsilon}(s,t) = R(s,t) + 2 \varepsilon U^{1}(s,t) + \ldots + 2\varepsilon^{m} U^{m}(s,t) \right\},
\end{equation}
where  $U^{j}: \Omega \to \R^3$ is a $C^{k}$ function for  $j \in\left\{1, 2, \ldots, m \right\}$, for some $k \in \N$.
The deformation $S_\varepsilon$ is an \textit{infinitesimal bending of $S$ of order $m$} if its first fundamental form $dR_{\varepsilon}^{2}$ satisfies the following condition:
\begin{equation} \nonumber
d R_{\varepsilon}^{2} = dR^{2} + o(\varepsilon^m), \ \text{as} \ \varepsilon \to 0.
\end{equation}
Since
$$dR_{\varepsilon}^{2} = dR^{2} + 4\varepsilon \left( dR \cdot dU^{1} \right) + \ds \sum_{j = 2}^{m} 4\varepsilon^{j} \left(dR \cdot dU^{j} + \ds \sum_{i =1}^{m-1} dU^{i} \cdot dU^{m-i} \right) + o(\varepsilon^m),$$
then $S_{\varepsilon}$ is an infinitesimal bending of order $m$ if and only if
\begin{equation} \label{The Godfather}
dR \cdot dU^{1} = 0 \ \text{and} \ dR \cdot dU^{j} = - \ds \sum_{i = 1}^{j-1} dU^{i} \cdot dU^{j-i}, \ \ j = 2, 3, \ldots, m.
\end{equation}
For each $j \in \left\{1, 2, \ldots, m \right\}$, set $U^{j}(s,t) = \left(u^{j}(s,t), v^{j}(s,t), w^{j}(s,t)\right)$.
Equation \eqref{The Godfather} can be written as
\begin{equation} \label{Dr. Strangelove}
\begin{cases}
x_{s} u^{j}_{s} + y_{s} v^{j}_{s} + z_{s} w^{j}_{s} =  F^{j}, \\
x_{s}  u^{j}_{t} + y_{s}  v^{j}_{t} + z_{s} w^{j}_{t} + x_{t}  u^{j}_{s} + y_{t} v^{j}_{s} + z_{t} w^{j}_{s} = G^{j},  \\
x_{t}  u^{1}_{t} + y_{t}  v^{1}_{t} + z_{t}  w^{1}_{t} =  H^{j},
\end{cases}
\end{equation}
with $F^{1} = G^{1} = H^{1} = 0$ and, when $j\geq 2$,
\begin{equation} \label{All the President's Men}
\begin{split}
F^{j} &= - \ds \sum_{i = 1}^{j-1} \left(u^{j-i}_{s}  u^{i}_{s} + v^{j-i}_{s}  v^{i}_{s} + w^{j-i}_{s} w^{i}_{s} \right), \\
G^{j} &= - \ds \sum_{i = 1}^{j-1} \left(u^{i}_{s}  u^{j-i}_{t} + v^{i}_{s}  v^{j-i}_{t} + w^{i}_{s}  w^{j-i}_{t} + u^{i}_{t}  u^{j-i}_{s} + v^{i}_{t} v^{j-i}_{s} + w^{i}_{t} w^{j-i}_{s} \right), \\
H^{j} &= - \ds \sum_{i = 1}^{j-1} \left(u^{j-i}_{t}  u^{i}_{t} + v^{j-i}_{t}  v^{i}_{t} + w^{j-i}_{t} w^{i}_{t} \right).
\end{split}
\end{equation}

The trivial bendings of $S$ are those generated through
the rigid motions of the underlying space $\R^3$. In particular, the first order trivial infinitesimal bendings of $S$ are given by $S^{A, B}_{\varepsilon} = \left\{R^{A,B}_{\varepsilon}(s,t) = R(s,t) + \varepsilon \left(A \times R(s,t) + B\right) \right\}$, where $A,B$ are constants in $\R^3$ and $\times$ denotes the vector product in $\R^3$.
A surface $S$ is said to be \textit{rigid} under infinitesimal bendings  if it admits only trivial infinitesimal bendings.

Let $N$ be the normal unit vector to $S$ given by $\ds \frac{R_{s} \times R_{t}}{\left\|R_{s} \times R_{t} \right\|}$ and
$e$, $f$, $g$, the coefficients of the second fundamental form of $S$:
\begin{equation} \label{Cidade de Deus}
e = R_{ss} \cdot N, \ \ \ \ \ \  f = R_{st} \cdot N, \ \ \ \ \  g = R_{tt} \cdot N.
\end{equation}
The Gaussian curvature of $S$ is:
\begin{equation} \label{One Flew Over the Cuckoo's Nest}
K(s,t) = \ds \frac{e g - f^{2}}{\left\|R_{s} \times R_{t} \right\|^{2}} .
\end{equation}
Throughout this work, except for the last section, we will assume that
 the Gaussian curvature of $S$ is nonnegative:
$K(s,t)\geq 0$, for all $(s,t) \in \Omega$.
\section{Surfaces with Nonnegative Curvature and Flat Points} \label{There Will Be Blood}

In this section, we  assume that the parametrization domain $\Omega\subset\R^2$ of $S$  is relatively compact
and that $K>0$ on $\overline{\Omega}$ except at finitely many points $p_1,\,\cdots\, p_n$
at which both principal curvatures vanish. We  assume throughout that $K$ vanishes uniformly only to a finite order
at each flat point $p_j$ (see \eqref{K-hypothesis} below)).
We prove that such a  surface $S$ admits nontrivial infinitesimal bendings
of any order. The idea is to use the complex vector field of asymptotic directions (see \cite{m1}, \cite{m4}) to reduce the
study of the bending equations into solving Bers-Vekua type equations (see also \cite{au} and \cite{u} for the
local deformation of surfaces near flat points).

Let $S$ be given by \eqref{Schindler's List} and
$e$, $f$, $g$,  the coefficients of its second fundamental form. We assume throughout this section the existence of $p_1,\,\,\cdots\, p_n\in \Omega$ such that
\begin{equation}\label{K-hypothesis}
\begin{split}
K(p)>0\quad\forall p\in\overline{\Omega}\setminus\{ p_1,\,\cdots\, , p_n\}\ \   \mathrm{and}\qquad \qquad\\
\mathrm{order}_{p_j}(K) =2\,\mathrm{order}_{p_j}(e)=2\,\mathrm{order}_{p_j}(g)\ \
\forall j\in\{1,\cdots ,n\}\, ,
\end{split}
\end{equation}
where $\mathrm{order}_{p}(F)$ denotes the order of vanishing of the function $F$ at the point $p$.

The field of asymptotic directions is given by
\begin{equation} \label{The Social Network}
 L = g(s,t) \ds \frac{\partial}{\partial s} + \lambda (s,t) \ds \frac{\partial}{\partial t}
\ \ \text{where} \ \ \lambda = -f + i \sqrt{eg - f^{2}} .
\end{equation}

\begin{proposition} \label{Wild Strawberries}
Let $S$ be a surface with nonnegative curvature given by \eqref{Schindler's List} and $L$ be
the vector field of asymptotic directions given by \eqref{The Social Network}.
For a solution  $U^{j}=(u^{j}, v^{j}, w^{j})$  of \eqref{Dr. Strangelove} with  $j \in \N$,  the
$\C$-valued function $h^{j} = LR \cdot U^{j}$ satisfies the equation
\begin{equation} \label{Raging Bull}
CLh^{j} = Ah^{j} - B \overline{h}^{j} + C \left[g^{2}F^{j} + g \lambda G^{j} + \lambda^{2} H^{j} \right],
\end{equation}
where
\begin{equation} \label{Marriage Story}
\begin{gathered}
A=(LR \times \overline{L}R)  (L^{2} R \times \overline{L}R), \ \ \
B=(LR \times \overline{L}R)  (L^{2} R \times LR),  \\
C=(LR \times \overline{L}R) (LR \times \overline{L}R).
\end{gathered}
\end{equation}
\end{proposition}

\begin{proof}
Define functions $\varphi^{j}$ and $\psi^{j}$ by
\begin{align}
\varphi^{j} &= R_{s} \cdot U^{j} = x_{s}u^{j} + y_{s} v^{j} + z_{s} w^{j},  \label{Blade Runner}\\
\psi^{j} &= R_{t} \cdot U^{j} = x_{t}u^{j} + y_{t} v^{j} + z_{t} w^{j}. \label{Blade Runner 2049}
\end{align}
Then
\begin{align}
\varphi_{s}^{j} &= R_{ss} \cdot U^{j} + F^{j} = x_{ss}u^{j} + y_{ss} v^{j} + z_{ss} w^{j} + F^{j}, \label{Star Wars}  \\
\psi_{t}^{j} &= R_{tt} \cdot U^{j} + H^{j} = x_{tt}u^{j} + y_{tt} v^{j} + z_{tt} w^{j} + H^{j}, \label{The Empire Strikes Back} \\
\varphi_{t}^{j} + \psi_{s}^{j}&= 2 R_{st} \cdot U^{j} + G^{j} = 2 \left[x_{st}u^{j} + y_{st} v^{j} + z_{st} w^{j} \right] + G^{j}. \label{Return of Jedi}
\end{align}
Let $R_{s} \times R_{t} = (\alpha_{1}, \alpha_{2}, \alpha_{3})  $ and $\alpha = \left\|R_{s} \times R_{t} \right\|$. It follows
 from \eqref{Blade Runner} and \eqref{Blade Runner 2049} that
\begin{align*}
\alpha_{3} u^{j} &= \alpha_{1}  w^{j} + \varphi^{j}  y_{t} - \psi^{j} y_{s},   \\
\alpha_{3} v^{j} &= \alpha_{2} w^{j} - \varphi^{j} x_{t} + \psi^{j} x_{s}.
\end{align*}

By using these expressions, we can rewrite  \eqref{Star Wars}, \eqref{The Empire Strikes Back} and \eqref{Return of Jedi} as
\begin{align*}
\alpha_{3} \varphi_{s}^{j} &= \begin{vmatrix}
x_{ss} & y_{ss} \\
x_{t}  & y_{t}
\end{vmatrix} \varphi^{j} -
\begin{vmatrix}
x_{ss} & y_{ss} \\
x_{s}  & y_{s}
\end{vmatrix} \psi^{j} + \alpha e w^{j} + \alpha_{3} F^{j}  \\
\alpha_{3} \psi_{t}^{j} &=
\begin{vmatrix}
x_{tt} & y_{tt} \\
x_{t}  & y_{t}
\end{vmatrix} \varphi^{j} -
\begin{vmatrix}
x_{tt} & y_{tt} \\
x_{s}  & y_{s}
\end{vmatrix} \psi^{j} + \alpha g w^{j} + \alpha_{3} H^{j} \\
\alpha_{3} (\varphi_{t}^{j} + \psi_{s}^{j})&=
\begin{vmatrix}
x_{st} & y_{st} \\
x_{t}  & y_{t}
\end{vmatrix} 2\varphi^{j} -
\begin{vmatrix}
x_{st} & y_{st} \\
x_{s}  & y_{s}
\end{vmatrix} 2 \psi^{j} + 2\alpha f w^{j} + \alpha_{3} G^{j}.
\end{align*}
We can eliminate the function $w^{j}$ in  the system above
(using Lemma $3.1$ of \cite{m2} and canceling $\alpha_3$) and reduce it to the following
system for $\varphi^j$ and $\psi^j$:
\begin{equation} \nonumber
g \varphi_{s}^{j} - e \psi_{t}^{j} = - \ds \frac{\left[ R_{t} \cdot (R_{ss} \times R_{tt}) \right]}{\alpha}  \varphi^{j} +\ds \frac{\left[ R_{s} \cdot (R_{ss} \times R_{tt}) \right]}{\alpha}  \psi^{j} + \left(gF^{j} - eH^{j} \right),
\end{equation}
\begin{align}
\nonumber
f (\varphi_{s}^{j} + \psi_{t}^{j}) - \ds \frac{e + g}{2} (\varphi_{t}^{j} + \psi_{s}^{j}) &= \ds \frac{\left[R_{t} \cdot (R_{st} \times (R_{ss} + R_{tt}) \right]}{\alpha}  \varphi^{j}  + f (F^{j}+ H^{j}) - \\
&- \ds \frac{\left[ R_{s} \cdot (R_{st} \times (R_{ss} + R_{tt}) \right]}{\alpha}  \psi^{j}  - \ds \frac{e + g}{2} G^{j}. \nonumber
\end{align}
We rewrite the system in a matrix form as
\begin{equation*}
\begin{split}
\begin{pmatrix} g & 0\\ f & -\frac{e+g}{2}\end{pmatrix}
\begin{pmatrix}\varphi^j\\ \psi^j\end{pmatrix}_s -
\begin{pmatrix} 0 & e\\ \frac{e+g}{2} & -f\end{pmatrix}
\begin{pmatrix}\varphi^j\\ \psi^j\end{pmatrix}_t
&= \underbrace{\begin{pmatrix}
	\xi_{11} & \xi_{12} \\
	\xi_{21} & \xi_{22}
	\end{pmatrix}}_{\Xi}
\begin{pmatrix}
\varphi^{j} \\
\psi^{j}
\end{pmatrix}+ \\
&+\begin{pmatrix}
gF^{j} - eH^{j} \\
f (F^{j}+ H^{j}) - \ds \frac{e + g}{2} G^{j}
\end{pmatrix}.
\end{split}
\end{equation*}
This system can be reduced further after multiplication  by $\ds \frac{2}{(e+g)}\begin{pmatrix}
-\ds \frac{e+g}{2} & 0 \\
-f &g                               \end{pmatrix}$ into 
\begin{equation} \label{Vertigo}
- g \begin{pmatrix}
\varphi^{j} \\
\psi^{j}
\end{pmatrix}_{s} -
\begin{pmatrix}
0 & -e \\
g & -2f
\end{pmatrix}  \begin{pmatrix}
\varphi^{j} \\
\psi^{j}
\end{pmatrix}_{t} = \Lambda  \begin{pmatrix}
\varphi^{j} \\
\psi^{j}
\end{pmatrix} +
 \begin{pmatrix}
e H^{j} - g F^{j}  \\
2f H^{j} -  g G^{j}
\end{pmatrix},
\end{equation}
where $\Lambda = \ds \frac{2}{(e+g)} \begin{pmatrix}
-\ds \frac{e+g}{2} & 0 \\
-f &g                               \end{pmatrix}  \Xi .$

Note that $\lambda$ is an eigenvalue for  the transpose of
$\begin{pmatrix}
0 & -e \\
g & -2f
\end{pmatrix}$ with eigenvector $\eta =\begin{pmatrix}
g \\
\lambda
\end{pmatrix} $.
After multiplying \eqref{Vertigo} by $\eta^{t}$ and using $\lambda^{2} + 2f \lambda  + eg = 0$, we get 
\begin{equation} \label{Boyhood}
 g  (g \varphi^{j}_{s} + \lambda \psi^{j}_{s}) + \lambda  (g \varphi^{j}_{t} + \lambda \psi^{j}_{t}) = - \eta^{t}  \Lambda  \begin{pmatrix}
\varphi^{j} \\
\psi^{j}
\end{pmatrix} + g^{2}F^{j}  + g \lambda G^{j} + \lambda^{2}H^{j}.
\end{equation}
Observe that $ \eta^{t}  \begin{pmatrix}
\varphi^{j} \\
\psi^{j}
\end{pmatrix} = g \varphi^{j} + \lambda \psi^{j} =  h^{j}$, so that $h^{j}_{s} = \eta^t  {\begin{pmatrix}
\varphi^{j} \\
\psi^{j}
\end{pmatrix}}_{s} +  \eta^t_s  \begin{pmatrix}
\varphi^{j} \\
\psi^{j}
\end{pmatrix}$, $ h^{j}_{t} = \eta^t  {\begin{pmatrix}
\varphi^{j} \\
\psi^{j}
\end{pmatrix}}_{t} + \eta^t_t  \begin{pmatrix}
\varphi^{j} \\
\psi^{j}
\end{pmatrix}$
and \eqref{Boyhood} becomes
\begin{equation} \label{American Beauty}
g  h_{s}^{j} + \lambda  h_{t}^{j} = g \eta^t_s \begin{pmatrix}
\varphi^{j} \\
\psi^{j}
\end{pmatrix} +
\lambda \eta^t_t  \begin{pmatrix}
\varphi^{j} \\
\psi^{j}
\end{pmatrix} - \eta^{t}  \Lambda  \begin{pmatrix}
\varphi^{j} \\
\psi^{j}
\end{pmatrix} + g^{2}F^{j}  + g \lambda G^{j} + \lambda^{2}H^{j}.
\end{equation}
Since $g$ is a real function, 
\begin{equation} \label{The Color Purple}
(\lambda - \overline{\lambda})g  \varphi^{j} = \lambda \overline{h}^{j} - \overline{\lambda} h^{j}, \ \ \ \ (\lambda - \overline{\lambda}) \psi^{j} = h^{j} - \overline{h}^{j}.
\end{equation}
Hence after multiplying \eqref{American Beauty} by $g(\lambda - \overline{\lambda})$, we get
\begin{equation} \label{Roma}
g(\lambda - \overline{\lambda}) Lh = Ph +  Q \overline{h} +
g(\lambda - \overline{\lambda}) \left(g^{2}F^{j}  + g \lambda G^{j} + \lambda^{2}H^{j} \right),
\end{equation}
where the coefficients $P$ and $Q$ are given by.
$$P = g(\lambda - \overline{\lambda}) \ds \frac{(L^{2}R \times \overline{L}R) \cdot (LR \times \overline{L}R)}{(LR \times \overline{L}R) \cdot (LR \times \overline{L}R)} = g(\lambda - \overline{\lambda}) \ds \frac{A}{C}, $$
$$Q= -g(\lambda - \overline{\lambda}) \ds \frac{(L^{2}R \times LR) \cdot (LR \times \overline{L}R)}{(LR \times \overline{L}R) \cdot (LR \times \overline{L}R)} = g(\lambda - \overline{\lambda}) \ds \frac{B}{C} $$
(see \cite{m2} for details).
This completes the proof of the proposition.
\end{proof}
\begin{remark}
A direct calculation gives
\begin{equation} \label {1917}
\begin{split}
L^{2}R \times LR &= (\lambda \cdot Lg - g \cdot L\lambda) (R_{s} \times R_{t})  + g^{3} (R_{ss} \times R_{s}) + \lambda^{3} (R_{tt} \times R_{t}) + \\
&+ g^{2} \lambda \left[(R_{ss} \times R_{t}) + 2 (R_{st} \times R_{s})\right] +  g \lambda^{2}\left[2 (R_{st} \times R_{t}) + (R_{tt} \times R_{s}) \right],  \\
L^{2}R \times \overline{L}R &= (\overline{\lambda} \cdot Lg - g \cdot L\lambda) (R_{s} \times R_{t}) + g^{3} (R_{ss} \times R_{s}) + \lambda |\lambda|^{2}  (R_{tt} \times R_{t}) + \\
&+ g^{2} \overline{\lambda} (R_{ss} \times R_{t}) + 2g^{2} \lambda (R_{st} \times R_{s}) + 2g|\lambda|^{2} (R_{st} \times R_{t}) + g \lambda^{2} (R_{tt} \times R_{s}),
\\
LR \times \overline{L}R&= -g (\lambda - \overline{\lambda}) (R_{s} \times R_{t}),
\end{split}
\end{equation}
which implies that $C = - 4g^{2} (eg - f^{2}) \left\|R_{s} \times R_{t} \right\|^{2}.$
\end{remark}
To continue, we need to understand the behavior of the coefficients $A$, $B$ and $C$ at the
flat points $p_{1}, \ldots, p_{n}$.
Since the order of contact of the surface $S$ with the tangent plane at the flat points
$p_j$ is $m_j\ge 3$,
by proceeding as in \cite{m1} we can find local polar coordinates $(r, \theta)$ centered in $p_{j}$ such that
\begin{equation}\label{Back to the Future}
\begin{gathered}
e = r^{m_{j}-2} e_{1}^{j}({\theta}) + r^{m_{j}-1} e_{2}^{j} (r, \theta),  \ \ \  f = r^{m_{j}-2} f_{1}^{j}({\theta}) + r^{m_{j}-1} f_{2}^{j} (r, \theta) , \\
g = r^{m_{j}-2} g_{1}^{j}({\theta}) + r^{m_{j}-1} g_{2}^{j} (r, \theta).
\end{gathered}
\end{equation}
The hypothesis \eqref{K-hypothesis} implies that
$e_{1}^{j}(\ta)g_{1}^{j}(\ta)-f_{1}^{j}(\ta)^2 >0$ for all $\ta\in\R$. This fact, associated to \eqref{The Social Network}, implies that
\begin{equation} \label{Back to the Future II}
\lambda = r^{m_{j}-2} \lambda_{1}^{j}({\theta}) + r^{m_{j}-1} \lambda_{2}^{j} (r, \theta).
\end{equation}
Furthermore the vector field $L$ can
be normalized and written as
\begin{equation} \label{One Million Dollar Baby}
L = \Xi (r, \theta)   r^{m_{j}-3} \left(\mu_{j}\ds \frac{\partial}{\partial \theta} - ir \ds \frac{\partial}{\partial r} \right),
\end{equation}
 where $\mu_{j} > 0$ is an invariant attached to $L$ and $\Xi \neq 0$ everywhere.

In these coordinates we have
\begin{equation} \label{Back to the Future III}
\begin{split}
Lg &= \Xi (r, \theta) r^{2m_{j} - 5} \zeta_{1}^{j} (\theta) + r^{2m_{j} - 4} \zeta_{2}^{j}(r, \theta); \\
L\lambda &= \Xi (r, \theta)   r^{2m_{j} - 5} \vartheta_{1}^{j} (\theta) + r^{2m_{j} - 4} \vartheta_{2}^{j}(r, \theta).
\end{split}
\end{equation} 
Using \eqref{1917}, \eqref{Back to the Future}, \eqref{Back to the Future II} and \eqref{Back to the Future III},  we deduce that 
\begin{align*}
L^{2}R \times LR &= \Xi (r, \theta) r^{3m_{j} - 7} \varsigma_{1}^{j}(\theta) (R_{s} \times R_{t})  +  r^{3m_{j} - 6} \varsigma_{2}^{j}(r, \theta); \\
L^{2}R \times \overline{L}R &= \Xi (r, \theta) r^{3m_{j} - 7} \varkappa_{1}^{j}(\theta) (R_{s} \times R_{t})  +  r^{3m_{j} - 6} \varkappa_{2}^{j}(r, \theta); \\
LR \times \overline{L}R&=  r^{2m_{j} - 4} \kappa_{1}^{j}(\theta) (R_{s} \times R_{t})  + r^{2m_{j} - 3} \kappa_{2}^{j}(r, \theta).
\end{align*}
Hence $A$, $B$, and $C$ can be written as:
\small{
\begin{equation}\label{ABC}\begin{array}{ll}
A &= (L^{2}R \times \overline{L}R)  (LR \times \overline{L}R) = \Xi (r, \theta) r^{5m_{j} - 11} \varrho_{1}^{j}(\theta) \left\|R_{s} \times R_{t} \right\|^{2} +   r^{5m_{j} - 10} \varrho_{2}^{j}(r, \theta);  \\
B &= (L^{2}R \times LR)  (LR \times \overline{L}R) =  \Xi (r, \theta) r^{5m_{j} - 11} \nu_{1}^{j}(\theta) \left\|R_{s} \times R_{t} \right\|^{2} +   r^{5m_{j} - 10} \nu_{2}^{j}(r, \theta); \\
C &=(LR \times \overline{L}R)  (LR \times \overline{L}R) = r^{4m_{j} - 8} \mu_{1}^{j}(\theta) \left\|R_{s} \times R_{t} \right\|^{2} +   r^{4m_{j} - 9} \mu_{2}^{j}(r, \theta).
\end{array}
\end{equation}
}
\normalsize

Since $\left\|R_{s} \times R_{t} \right\|^{2}$ is always strictly positive, we infer from \eqref{ABC} that
\begin{equation} \label{Eternal Sunshine of the Spotless Mind}
\begin{split}
\ds \frac{A}{C}(r, \theta) &= \Xi (r, \theta) r^{m_{j} - 3} a_{1}^{j}(\theta) + r^{m_{j} - 2} a_{2}^{j}(r, \theta); \\ 
\ds \frac{B }{C} (r, \theta) &= \Xi (r, \theta) r^{m_{j} - 3} b_{1}^{j}(\theta) + r^{m_{j} - 2} b_{2}^{j}(r, \theta).
\end{split}
\end{equation}
The following result about the first integral of $L$ (proved in \cite{m1}) will be used.

\begin{lemma}\cite{m1} \label{The Seventh Seal}
There exists an injective function $Z: \overline{\Omega} \to \C$ satisfying the following conditions:
\begin{enumerate} 
\item $Z$ is $C^{\infty}$ on $\overline{\Omega} \setminus \left\{p_{1}, p_{2}, \ldots, p_{n} \right\}$.
\item $LZ = 0$ on $\overline{\Omega}$.
\item For every $j=1, 2,, \ldots, n$, there exists $\mu_{j} > 0$ and polar coordinates $(r, \theta)$ centered at $p_j$ such that
 \begin{equation} \label{Her}
 Z(r, \theta) = Z(0,0) + r^{\mu_{j}} \cdot e^{i \theta} + O(r^{2 \mu_{j}})
 \end{equation}
in a neighborhood of $p_{j}$.
\end{enumerate}
\end{lemma}
We use the first integral $Z$ of $L$ given by \eqref{Her} to transform the equation \eqref{Raging Bull} to a Bers-Vekua type equation. The following notation will be used:
for each  $\ell \in \left\{1, 2, \ldots, n \right\}$, let $\zeta_{\ell} = Z(p_{\ell})$ and
$D(\zeta)=\prod_{\ell = 1}^{n}(\zeta - \zeta_{\ell})$.
The pushforward via $Z$ of a function $f$ defined in $\Omega$ will be denoted $\widetilde{f}$:
$\, \widetilde{f}=f\circ Z^{-1}$.

\begin{proposition} \label{On the Waterfront}
Let $Z$ as in Lemma \ref{The Seventh Seal}.
A function $h^{j}$ is a solution of $\eqref{Raging Bull}$ in $\Omega$ if and only if
its $Z$-pushforward $\widetilde{h}^{j}$ satisfies the following equation in $Z(\Omega)$:
\begin{equation}\label{2001}
 \ds \frac{\partial \widetilde{h}^{j}}{\partial \overline{\zeta}}=
 \ds \frac{P(\zeta)}{D(\zeta)} \widetilde{h}^{j} + \ds \frac{Q(\zeta)}{D(\zeta)} \overline{\widetilde{h}}^{j} + \ds \frac{1}{D(\zeta)}
  \left( S_{1}(\zeta) \widetilde{F}^{j} +  S_{2}(\zeta) \widetilde{G}^{j} +  S_{3}(\zeta) \widetilde{H}^{j}  \right),
\end{equation}
with the following conditions being satisfied: 
\begin{enumerate}[label=\arabic*), leftmargin=*]
\item $P,Q\in L^{\infty}(Z(\Omega))\, \cap\, C^{\infty}\left(Z(\Omega) \setminus
\left\{\zeta_{1}, \zeta_{2}\, \ldots, \zeta_{n} \right\} \right)$; morevoer, we are able to write in a small neighborhood of each $\zeta_{j}$ the coefficients as
\begin{equation*}
P(\zeta) = \chi_{1,j}(\sigma)+\rho^{\epsilon_j}\kappa_{j}^{1}(\rho, \sigma), \  Q(\zeta) = \chi_{2, j}(\sigma) +\rho^{\epsilon_j}\kappa_{j}^{2}(\rho, \sigma),
\end{equation*}
where $\zeta = \rho e^{i\sigma}$, both $\chi_{1, j}$ and $\chi_{2, j}$ are $2\pi$-periodic and $\epsilon_{j} > 0$, for every $j$.

\item $S_{1}, S_{2}, S_{3}$ are bounded in $Z(\Omega)$ and
vanish  to order $\ds \frac{m_{\ell}-1}{\mu_{\ell}}$ at $\zeta_{\ell}$, for every $\ell$, where $m_\ell$ and $\mu_\ell$ are
the positive numbers associated with the point $p_\ell$.
\end{enumerate}  
\end{proposition}

\begin{proof} The $Z$-pushforward of \eqref{Raging Bull} gives
\begin{equation}\label{Corona}
\widetilde{\left(L \overline{Z}\right)}\, \frac{\partial \widetilde{h}^{j}}{\partial \overline{\zeta}} =
\frac{\widetilde{A}}{\widetilde{C}}\,\widetilde{h}^{j}-\frac{\widetilde{B}}{\widetilde{C}}\,\overline{\widetilde{h}^{j}}
+\widetilde{M},
\end{equation}
where 
\begin{equation} \label{The Imitation Game} 
M=g^2F^j+g\lambda G^j+\lambda^2H^j.
\end{equation}
Since the vector field $L$ is elliptic outside the points $\zeta_j$ and $C$ does not
vanish outside these points, \eqref{Corona} has the form \eqref{2001} in $Z(\Omega)\setminus\{\zeta_1,\cdots ,\zeta_n\}$. Let us verify the proposition near each point $\zeta_j$. It follows from \eqref{One Million Dollar Baby} and \eqref{Her} that
\begin{equation} \label{The Aviator}
\overline{L}Z = -2i \Xi(r, \theta) \left[  r^{m_{j} -3 + \mu_{j}} e^{-i \theta} + r^{m_{j}-3} O(r^{2 \mu_{j}}) \right]
\end{equation}

The next step is to understand how composition with $Z^{-1}$ acts over each term. By assuming that $Z(0, 0) = 0$ and setting $\rho = |Z|$, it follows from \eqref{Her} that we are able to write $\rho(r, \theta) = r^{\mu_{j}} \left(1+ r^{\mu_{j}} J_{1}^{j}(r, \theta) \right)^{1/2},$
where $J_{1}^{j}(r, \theta)$ is a continuous and bounded function. By the binomial theorem, if $r$  is sufficiently small we have
\begin{equation*}
\rho(r, \theta) = r^{\mu_{j}} \left(1 + r^{\mu_{j}} J_{2}^{j}(r, \theta) \right) = r^{\mu_{j}} + r^{2\mu_{j}} J_{2}^{j}(r, \theta) = y + y^{2} J_{3}^{j}(y, \theta),
\end{equation*}
if we denote $y = r^{\mu_{j}}$. As a consequence of the proof of Lemma \ref{The Seventh Seal}, $Z$ is a $C^{1}$ function in terms of $y$ and $\theta$. Thus, by possibly taking $r$ even smaller we are able to solve $\rho$ in terms of $y$, which allows us to deduce that $r^{\mu_{j}} = \rho + \rho^{2}K_{1}(\rho, \theta)$, with $K_{1}$ continuous and bounded. This implies that
\begin{equation} \label{Tropa de Elite II} 
r = \rho^{1/\mu_{j}} \left(1 + \rho K_{1}^{j}(\rho, \theta) \right)^{1/\mu_{j}}  = \rho^{1/\mu_{j}} + \rho^{1 + 1/\mu_{j}} K_{2}^{j} (\rho, \theta). 
\end{equation}

We deduce from \eqref{Back to the Future}, \ref{Back to the Future II}, \eqref{Eternal Sunshine of the Spotless Mind}, \eqref{The Imitation Game} and \eqref{Tropa de Elite II} that, for some $\varepsilon_{j} > 0$,
\begin{equation} \label{Adaptation}
\begin{split}
\widetilde{M} (\rho, \sigma) &= \widetilde{\Xi} (\rho, \sigma) \left[ \rho^{\frac{2m_{j}-4}{\mu_{j}}} \gamma_{1}(\rho, \sigma) \widetilde{F^j} +  \rho^{\frac{2m_{j}-4}{\mu_{j}}} \gamma_{2}(\rho, \sigma) \widetilde{G^j} + \rho^{\frac{2m_{j}-4}{\mu_{j}}} \gamma_{3}(\rho, \sigma) \widetilde{H^j} \right], \\
\ds \frac{\widetilde{A}}{\widetilde{C}}(\rho, \sigma) &= \widetilde{\Xi} (\rho, \sigma) \left[ \rho^{\frac{m_{j} - 3}{\mu_{j}}} \alpha_{1}^{j}(\sigma) + \rho^{\frac{m_{j} - 3}{\mu_{j}} + \varepsilon_{j}} \alpha_{2}^{j}(\rho, \sigma) \right], \\ 
\ds \frac{\widetilde{B}}{\widetilde{C}}(\rho, \sigma) &= \widetilde{\Xi} (\rho, \sigma) \left[ \rho^{\frac{m_{j} - 3}{\mu_{j}}} \beta_{1}^{j}(\sigma) + \rho^{\frac{m_{j} - 3}{\mu_{j}} + \varepsilon_{j}} \beta_{2}^{j}(\rho, \sigma) \right].
\end{split}
\end{equation}
By using expression \eqref{Tropa de Elite II} in \eqref{The Aviator}, we obtain
\begin{equation} \label{Lord of War}
\widetilde{L \overline{Z}}(\rho, \sigma) =  \widetilde{\Xi} (\rho, \sigma) \left[ \rho^{ \frac{m_{j} -3}{\mu_{j}} + 1} \psi_{1}^{j}(\sigma) + \rho^{\frac{m_{j} -3}{\mu_{j}} + 2} \psi_{2}^{j}(\rho, \theta) \right].
\end{equation}
It follows from \eqref{Adaptation} and \eqref{Lord of War} that \eqref{Corona} can be rewritten as 
\begin{align}
\frac{\partial \widetilde{h}^{j}}{\partial \overline{\zeta}} &= \left[\left(\ds \frac{\chi_{1}^{j}(\sigma) + \rho^{\epsilon_{j}} \chi_{2}^{j}(\rho, \theta)}{\rho e^{i \sigma}} \right)\widetilde{h}^{j} + \left(\ds \frac{\kappa_{1}^{j}(\sigma) + \rho^{\epsilon_{j}} \kappa_{2}^{j}(\rho, \theta)}{\rho e^{i \sigma}} \right)\overline{\widetilde{h}^{j}} \right] + \nonumber \\
&+ \left[\ds \frac{\rho^{\frac{m_{j} - 1}{\mu_{j}} } \iota_{1}^{j}(\rho, \sigma) \widetilde{F^j}}{\rho e^{i \sigma}} + \frac{\rho^{\frac{m_{j} - 1}{\mu_{j}} } \iota_{2}^{j}(\rho, \sigma) \widetilde{G^j}}{\rho e^{i \sigma}} + \frac{\rho^{\frac{m_{j} - 1}{\mu_{j}} } \iota_{3}^{j}(\rho, \sigma) \widetilde{H^j}}{\rho e^{i \sigma}}   \right], \label{Cries and Whispers}  
\end{align}
where $\epsilon_{j} = \min \left\{\varepsilon_{j}, 1 \right\}$. Observe that  $\rho e^{i\sigma} = \zeta - \zeta_{j}.$ Hence, in order to obtain an expression as in \eqref{2001}, it is sufficient to multiply both the numerator and the denominator of each term in the right hand-side of \eqref{Cries and Whispers}   by $D_{j}(\zeta)= \ds \prod_{\underset{\ell \neq j}{\ell = 1, \ldots, n}}(\zeta - \zeta_{\ell}), $ which finalizes the proof. 
\end{proof}
We can now proceed to the main result of the section.

\begin{theorem} \label{Prisoners}
Let $S\subset\R^3$ be a smooth surface given by \eqref{Schindler's List} with curvature $K$ satisfying \eqref{K-hypothesis}.
For every $k, m \in \N$,
there exist functions
$$U^{1}, U^{2}, \ldots, U^{m}: \Omega \to \R^{3}\quad\textrm{with}\quad
U^j\in C^k(\Omega)\,\cap\, C^\infty(\Omega \setminus \left\{p_{1}, p_{2}, \ldots, p_{n} \right\}),$$
such that each $U^j$ vanishes to order $\ge k $ at each point $p_1,\, \cdots\, p_n$ and such that
the deformation surface $S_\epsilon$ given by \eqref{Sunset Boulevard} is a nontrivial
infinitesimal bending of $S$ of order $m$.
\end{theorem}

\begin{proof}
We prove this result by induction on the order $m$;  in general the idea is to solve \eqref{2001} and show it has a solution $\widetilde{h}^{j}$ that vanishes, to any prescribed order, at the points $\zeta_1,\,\cdots\, ,\zeta_n$. Then use $h^j=\widetilde{h}^j\circ Z$ to recover
the field $U^j$ through $h^j=LR\cdot U^j$ (Proposition \ref{Wild Strawberries}). 
 
Let $j= 1$; then $F^{1} = G^{1} = H^{1} = 0$, which turns the expression \eqref{2001} into
 \begin{equation} \label{The Dark Knight}
 	\ds \frac{\partial \widetilde{h}^{1}}{\partial \overline{\zeta}}=
 	\ds \frac{P(\zeta)}{D(\zeta)} \widetilde{h}^{1} + \ds \frac{Q(\zeta)}{D(\zeta)} \overline{\widetilde{h}}^{1}.
 \end{equation}
Applying \emph{Theorem} $2.3$ of \cite{lvm}, for any $l \in \N$ we are able to find a non-trivial solution for \eqref{The Dark Knight} in $ C^{l}(Z(\Om)) \cap C^{\infty}(Z(\Om)\backslash \{\zeta_1,\cdots ,\zeta_n\}),$ such that it vanishes to order $\ge l $ at each point $\zeta_1,\, \cdots\, \zeta_n$. Therefore, as a consequence of Proposition \ref{On the Waterfront},  $h^1=\widetilde{h}^1\circ Z$  solves \eqref{Raging Bull}  and vanishes to order $l \mu_i$ at the points $p_1, \ldots, p_{n}$.

Next we construct the field $U^1$ from $h^1$. As a consequence of  \eqref{The Color Purple}, we have  $ \varphi^{1} = \ds \frac{\lambda \overline{h}^{1} - \overline{\lambda} h^{1}}{(\lambda - \overline{\lambda})g }$ and  $\psi^{1} = \ds \frac{h^{1} - \overline{h}^{1}}{(\lambda - \overline{\lambda})}. $ These functions would be well defined at the flat points $p_1,\cdots ,p_n$ provided that $h^1$ vanishes
 to higher orders than those the vanishing of the denominators. This is indeed the case when
$ (l \cdot \mu_{i}) - (m_{i} -2) \geq 0$, since both $g$ and $\lambda - \overline{\lambda}$ have order of vanishing $m_i-2$ at $p_i$ (see \ref{Back to the Future} and \ref{Back to the Future II}).
 
 The field $U^1=(u^1,v^1,w^1)$ is related to $\varphi^1$ and $\psi^1$ via relations
 \eqref{Blade Runner}, \eqref{Blade Runner 2049} and \eqref{Star Wars}, which form the system
 \begin{equation} \nonumber
 	\begin{pmatrix}
 		\varphi^{1} \\
 		\psi^{1} \\
 		\varphi_{s}^{1} 
 	\end{pmatrix} =
 	\begin{pmatrix}
 		x_{s} & y_{s} & z_{s} \\
 		x_{t} & y_{t} & z_{t} \\
 		x_{ss} & y_{ss} & z_{ss}
 	\end{pmatrix}
 	\begin{pmatrix}
 		u^{1} \\
 		v^{1} \\
 		w^{1}
 	\end{pmatrix}=M(s,t)\begin{pmatrix}
 		u^{1} \\
 		v^{1} \\
 		w^{1}
 	\end{pmatrix}\, .
 \end{equation}
The determinant of the matrix $M$ is given by $R_{ss} \cdot (R_{s} \times R_{t}) = 
\left\|R_{s} \times R_{t} \right\| e.$ Thus
\begin{equation} \label{Kramer vs Kramer} 
 	\begin{split}
 		u^{1} &= \ds \frac{1}{\left\|R_{s} \times R_{t} \right\| e}  \left( \psi^{1} \begin{vmatrix}
 			y_{s} & z_{s} \\ y_{ss} & z_{ss}
 		\end{vmatrix}  -  \varphi^{1} \begin{vmatrix}
 			y_{t} & z_{t} \\ y_{ss} & z_{ss}
 		\end{vmatrix} + \varphi_{s}^{1} \begin{vmatrix}
 			y_{s} & z_{s} \\
 			y_{t} & z_{t}
 		\end{vmatrix}
 		\right),  \\
 		v^{1} &= \ds \frac{1}{\left\|R_{s} \times R_{t} \right\| e}  \left(- \varphi^{1}  \begin{vmatrix}
 			x_{t} & z_{t} \\ x_{ss} & z_{ss}
 		\end{vmatrix}  + \psi^{1}  \begin{vmatrix}
 			x_{s} & z_{s} \\ x_{ss} & z_{ss}
 		\end{vmatrix} - \varphi_{s}^{1} \begin{vmatrix}
 			x_{s} & z_{s} \\
 			x_{t} & z_{t}
 		\end{vmatrix}
 		\right),  \\
 		w^{1} &= \ds \frac{1}{\left\|R_{s} \times R_{t} \right\| e}  \left( \varphi^{1}  \begin{vmatrix}
 			x_{t} & y_{t} \\ x_{ss} & y_{ss}
 		\end{vmatrix}  - \psi^{1}  \begin{vmatrix}
 			x_{s} & y_{s} \\ x_{ss} & y_{ss}
 		\end{vmatrix} + \varphi_{s}^{1}  \begin{vmatrix}
 			x_{s} & y_{s} \\
 			x_{t} & y_{t}
 		\end{vmatrix}
 		\right).
 	\end{split}
 \end{equation}
 These functions are well defined if $\varphi^{1}, \psi^{1}$, $\varphi_{s}^{1}$ vanish to an order
 greater than $m_{i} - 2$ at each $p_{i}$. In this case $u^1$, $v^1$, and $w^1$ vanish to order
 \begin{equation}\label{Fences}
 	[l \mu_{i} - (m_{i} - 1)] - (m_{i} - 2) = l \mu_{i} - 2m_{i} + 3. 
 \end{equation}

Now consider $m^{\star} = \max \left\{m_{i}, \ i = 1, \ldots, n \right\}$ and $\mu_\star = \min \left\{\mu_{i}, \ i = 1, \ldots, n \right\}.$	Given $k \in \N$, by choosing $l \geq \ds \frac{2m^{\star} - 3  + k}{\mu_{\star}}$, we  have $l \mu_{i} - 2m_{i} + 3 \geq k$ for each $i \in \left\{1, 2, \ldots, n \right\}$.  This implies (from \eqref{Fences}) that $U^1=(u^1,v^1,w^1)$ vanishes to order greater than $ k$ at each $p_{i}$. Hence $U^1\in C^k(\Omega)\,\cap\, C^\infty(\Omega \setminus \left\{p_{1}, p_{2}, \ldots, p_{n} \right\})$. 

It remains to show that $U^{1}$ is not trivial. Suppose by contradiction that
\begin{equation*}
U^{1} = A \times R(s,t) + B, \ \ \ A, B \in \R^{3}.
\end{equation*}
Then $U_{s}^{1}(p_{1}) =  A \times R_{s}(p_{1}) = 0$ and $U_{t}^{1}(p_{1}) =  A \times R_{t}(p_{1}) = 0$, which implies that $A = 0$. Since $U^{1}$ vanishes at $p_{0}$, then $B = 0$ and $U^{1} \equiv 0$, which ends the first case.

Suppose next that the statement holds to order up to $j-1$; then for every $\ell \in \N$ there exist functions $U^{1}, U^{2}, \ldots, U^{j-1}\,\in\, C^\infty(\Omega\setminus\{p_1,\cdots ,p_n\},\R^3)\,\cap\, C^\ell(\Omega,\R^3),$
such that each field $U^r$ vanishes to order $\ell$ at each point $p_i$ and $$R^{j-1}_\epsilon (s,t)=R(s,t)+2\sum_{r=1}^{j-1}\epsilon^r U^r(s,t)$$ 
is an infinitesimal bending of order $j-1$ of $S$.

Since $U^1,\,\cdots ,\, U^{j-1}$ vanish to order $\ell$ at the points $p_i$, the functions $F^r$, $G^r$, and $H^r$ given by \eqref{All the President's Men} vanish to order $2 \ell - 2$. Hence their $Z$-Pushforwards
$\widetilde{F}^r$, $\widetilde{G}^r$, and $\widetilde{H}^r$ vanish to order
$\ds\frac{2\ell -2}{\mu_i}$ at the points $\zeta_i=Z(p_i)$.  Thus  (by Proposition \ref{On the Waterfront}) the nonhomogeneous term of equation \eqref{2001} vanishes to order $\rho_i = \ds \frac{2 \ell + m_{i} - 3}{\mu_{i}}  - 1$
at $\zeta_i$. 

By taking $\ell$ sufficiently large and applying \emph{Theorem} $2.1$ of \cite{lvm}, we obtain a solution $\widetilde{h}^j$ of \eqref{2001} that vanishes to order $q$ at each point $\zeta_1,\cdots ,\zeta_n$, for some $q$ that will be chosen later. 
Therefore $h^j=\widetilde{h}^j\circ Z$ solves \eqref{Raging Bull}  and vanishes to order
$q\mu_i$ at the point $p_i$, for every $i=1,\cdots, n$. Now we construct the field $U^j$ from $h^j$. Once again $ \varphi^{j} =  \ds \frac{\lambda \overline{h}^{j} - \overline{\lambda} h^{j}}{(\lambda - \overline{\lambda})g }$, $\psi^{j} = \ds \frac{h^{j} - \overline{h}^{j}}{(\lambda - \overline{\lambda})} $ and  these functions would be well defined  if $q  \mu_{i} - (m_{i} -2) \geq 0$. Once again the field $U^j=(u^j,v^j,w^j)$ is related to $\varphi^j$ and $\psi^j$ via relations
\eqref{Blade Runner}, \eqref{Blade Runner 2049} ,\eqref{Star Wars} and in this case
\begin{align*}
u^{j} &= \ds \frac{1}{\left\|R_{s} \times R_{t} \right\| e}  \left( \psi^{j} \begin{vmatrix}
y_{s} & z_{s} \\ y_{ss} & z_{ss}
  \end{vmatrix}  -  \varphi^{j} \begin{vmatrix}
y_{t} & z_{t} \\ y_{ss} & z_{ss}
  \end{vmatrix} + (\varphi_{s}^{j} - F^{j}) \begin{vmatrix}
y_{s} & z_{s} \\
y_{t} & z_{t}
\end{vmatrix}
 \right),  \\
 v^{j} &= \ds \frac{1}{\left\|R_{s} \times R_{t} \right\| e}  \left(- \varphi^{j}  \begin{vmatrix}
x_{t} & z_{t} \\ x_{ss} & z_{ss}
  \end{vmatrix}  + \psi^{j}  \begin{vmatrix}
x_{s} & z_{s} \\ x_{ss} & z_{ss}
  \end{vmatrix} - (\varphi_{s}^{j} - F^{j}) \begin{vmatrix}
x_{s} & z_{s} \\
x_{t} & z_{t}
\end{vmatrix}
 \right),  \\
 w^{j} &= \ds \frac{1}{\left\|R_{s} \times R_{t} \right\| e}  \left( \varphi^{j}  \begin{vmatrix}
x_{t} & y_{t} \\ x_{ss} & y_{ss}
  \end{vmatrix}  - \psi^{j}  \begin{vmatrix}
x_{s} & y_{s} \\ x_{ss} & y_{ss}
  \end{vmatrix} + (\varphi_{s}^{j} - F^{j}) \begin{vmatrix}
x_{s} & y_{s} \\
x_{t} & y_{t}
\end{vmatrix}
 \right).
\end{align*}
These functions will be well defined if $\varphi^{j}, \psi^{j}$, $\varphi_{s}^{j}$ and $F^{j}$ vanish to an order
greater than $m_{i} - 2$ at each $p_{i}$. In this situation $u^j$, $v^j$, and $w^j$ vanish to order
\begin{equation} \label{The Pianist}
r_{i} =\min \left\{[q \mu_{i} - (m_{i} - 1)], 2 \ell - 2  \right\} - (m_{i} - 2).
\end{equation}

Let $m^{\star}$ and $\mu_\star$ as in the previous case.  Given $k\in\N$, take $\ell$  large enough so 
\begin{equation*}
\text{$2 \ell \geq k + m^{\star}$ \ \ and  \ \ $q \geq \ds \frac{k + 2m^{\star} - 3}{\mu_{\star}}$.}
\end{equation*}
It follows from \eqref{The Pianist} that for such choices of $\ell$ and $q$, we have $r_{i} \ge k$  for each $i$ and the field $U^{j} = (u^{j}, v^{j}, w^{j})$  vanishes to order $k$ at each $p_{i}$. This completes the proof.
\end{proof}
\
\section{Infinitesimal Bendings of Graphs of  Homogeneous Functions}

In this section we study infinitesimal bendings of surfaces given as graphs of homogeneous functions. We prove that any such generic surface with nonnegative curvature has nontrivial infinitesimal bendings of high orders. The surfaces considered here are given by
\begin{equation} \label{Four Weddings and a Funeral}
S = \left\{R(s,t) = \left(s, t, z(s,t) \right); \ (s, t) \in \R^{2}  \right\},
\end{equation}
where $z$ is a nonnegative homogeneous function of order $m\ge 2$
(with $m$ not necessarily an integer), such that $z\in C^\infty(\R^2\setminus \{0\})$. We
assume that $S$ has nonnegative curvature and no asymptotic curves.
The Gaussian curvature of $S$ is
$$K(s,t)=\frac{z_{ss}z_{tt}-z_{st}^2}{(1+z_s^2+z_t^2)^2}\, .$$

It is more convenient here to use polar coordinates $s=r\cos\ta$, $t=r\sin\ta$, so that
$z = r^m P(\theta)$, with $P(\theta) \in C^{\infty}(S^{1})$. The assumption that $S$ has no asymptotic curves implies that  $P(\theta) > 0$ for all $\ta$ (if $P(\ta_0)=0$, then the ray $\ta=\ta_0$ would be an asymptotic curve).
The second derivatives of $z$ in polar coordinates are:
\begin{equation} \nonumber
\begin{split}
z_{ss} &= r^{m-2} \left[(m-1)m \cos^{2}\theta P - 2 (m-1) \sin \theta \cos \theta P' + m \sin^{2} \theta P + P'' \sin^{2} \theta \right], \\
z_{st} &= r^{m-2} \left[(m-2)m \sin \theta \cos \theta P + (m-1) \cos (2\theta) P' - P'' \sin \theta \cos \theta \right], \\
z_{tt} &= r^{m-2} \left[(m-1)m \sin^{2}\theta P + 2 (m-1) \sin \theta \cos \theta P' + m \cos^{2} \theta P + P'' \cos^{2} \theta \right]. \\
\end{split}
\end{equation}
It follows that
\begin{equation} \nonumber
z_{ss}z_{tt} - z_{st}^{2} = r^{2m-4} (m-1) \left[m^{2} P^{2}(\theta) + m P (\theta) P''(\theta) - (m-1) P'(\theta)^{2} \right].
\end{equation}

From now on we will assume that $K\ge 0$ and $K$ vanishes on at most a finite number of rays $\ta=\ta_1,\,\cdots ,\ta_\ell$.
This is equivalent to
\begin{equation} \label{Eyes Wide Shut}
\left[m^{2} P^{2}(\theta) + m P (\theta) P''(\theta) - (m-1) P'(\theta)^{2} \right] > 0, \quad
 \forall \theta \notin \{\ta_1,\cdots ,\ta_\ell\}.
\end{equation}
We write in a more convenient form the equation related to the bending fields $U^j=(u^j,v^j,w^j)$.
The functions $\varphi^j$ and $\psi^j$ related to $U^j$  by
\eqref{Blade Runner} and \eqref{Blade Runner 2049} take the form
\begin{equation} \label{Citizen Kane}
\varphi^{j} = u^{j} + z_{s} w^{j},  \quad  \psi^{j} = v^{j} + z_{t} w^{j}\, .
\end{equation}
In this situation   equation  \eqref{Vertigo} leads to the following system in polar coordinates
\begin{equation} \label{Forrest Gump}
\ds \frac{1}{r} \begin{pmatrix}
 \varphi^{j}  \\
 \psi^{j}
 \end{pmatrix}_{\theta} = A(\ta)  \begin{pmatrix}
 \varphi^{j}  \\
 \psi^{j}
 \end{pmatrix}_{r} + T(\ta)\begin{pmatrix}
 F^j\\ G^j\\ H^j\end{pmatrix},
\end{equation}
where $A(\ta) = \ds\frac{1}{(m^2-m)r^{m-2}P(\ta)}
\begin{pmatrix}
\alpha_{11}(r,\ta) & \alpha_{12}(r,\ta)\\ \alpha_{21}(r,\ta) & \alpha_{22}(r,\ta)\
\end{pmatrix}$,
with
\begin{equation*}
\begin{split}
\alpha_{11}(r,\theta) &= {\left(z_{tt} - z_{ss} \right) \sin \theta \cos \theta + 2 z_{st} \cos^{2} \theta}, \ \ \ \ \alpha_{12}(r,\theta) =  {- z_{ss}}, \\
\alpha_{21}(r,\theta) &= {z_{tt}}, \ \ \ \
\alpha_{22}(r,\theta) = {\left(z_{tt} - z_{ss} \right) \sin \theta \cos \theta - 2 z_{st} \sin^{2} \theta},
\end{split}
\end{equation*}
and $T(\ta)$ is the matrix $T(\ta)=\ds \frac{1}{(m^2-m)r^{m-2}P(\ta)}\begin{pmatrix}
\beta_{1}(r,\ta) & \beta_{2}(r,\ta) &\beta_3(r,\ta)\\
\gamma_{1}(r,\ta) & \gamma_{2}(r,\ta) &\gamma_3(r,\ta)
\end{pmatrix}$, 
with
\begin{equation}\label{Uone}
\begin{array}{ll}
\beta_{1}(r,\theta) ={(-2 z_{st} \cos \theta - z_{tt} \sin \theta)}, &
\gamma_{1}(r,\theta) = {- z_{tt} \cos \theta},\\
\beta_{2}(r,\theta) = { z_{ss} \cos \theta}, &
\gamma_{2}(r,\theta) ={- z_{tt} \sin \theta}, \\
\beta_{3}(r \theta)=  { z_{ss} \sin \theta}, &
\gamma_{3}(r,\theta) = {z_{ss} \cos \theta + 2 z_{st} \sin \theta}.
\end{array}
\end{equation}
Note that it follows from the homogeneity of $z$ that the matrices $A$ and $T$ are independent of the variable $r$.

Now we use the change of variable $\rho =rP(\ta)^{1/m}$
to transform \eqref{Forrest Gump} into
\begin{equation} \label{Patton}
\ds \frac{1}{\rho} \begin{pmatrix}
 \varphi^{j}  \\
 \psi^{j}
 \end{pmatrix}_{\theta}  = \Lambda(\ta)
\begin{pmatrix}
 \varphi^{j}  \\ \psi^{j}\end{pmatrix}_{\rho} + \widetilde{T}(\ta)
 \begin{pmatrix}\widetilde{F}^j(\rho,\ta)\\\widetilde{G}^j(\rho,\ta)\\\widetilde{H}^j(\rho,\ta)
 \end{pmatrix}\, ,
\end{equation}
where $\Lambda(\ta)=\ds \frac{1}{(m^2-m)r^{m-2}P(\ta)}\begin{pmatrix}
 z_{st} & -z_{ss}\\ z_{tt} & -z_{st}
 \end{pmatrix}$
has trace 0,
\begin{equation} \label{Scarface}
 \widetilde{T}(\ta) =\frac{T(\ta)}{P(\ta)^{1/m}}\quad\textrm{and}\quad
 \begin{pmatrix}
 \widetilde{F}^j(\rho,\ta)\\
\widetilde{G}^j(\rho,\ta)\\
\widetilde{H}^j(\rho,\ta)\end{pmatrix}=
\begin{pmatrix}
 {F}^j\left(\frac{\rho}{P(\ta)^{1/m}},\ta\right)\\
 {G}^j\left(\frac{\rho}{P(\ta)^{1/m}},\ta\right)\\
 {H}^j\left(\frac{\rho}{P(\ta)^{1/m}},\ta\right)
\end{pmatrix}\, .
\end{equation}

\begin{remark}
For $j=1$, $F^1=G^1=H^1=0$, and the homogenous system \eqref{Patton} is studied in
\cite{m3}. It is proved that for every $n\in\N$, the surface $S$ has nontrivial first order
infinitesimal bendings fields $U^1\in C^n(\R^2)$ of the form
\begin{equation} \label{The Terminator}
U^{1}(r, \theta) =
 \left(r^{\lambda_{p}}a^{1}(\theta), r^{\lambda_{p}}b^{1}(\theta), r^{\lambda_{p} + 1 - m} c^{1}(\theta)  \right),
\end{equation}
with $a^{1}, b^{1}, c^{1}\in C^{\infty}(S^{1})$ and $\lambda_{p}$ an eigenvalue of the system
\begin{equation} \label{The Hurt Locker}
X'{\theta)} = \lambda \Lambda(\theta) X(\theta).
\end{equation}
\end{remark}

With the first bending $U^1$ as given by \eqref{The Terminator}, it follows from \eqref{All the President's Men}
that
\begin{equation} \label{Indiana Jones and the Last Crusade}
\begin{split}
F^{2}(r,\ta) &= r^{2\lambda_{p} - 2} f_{1} (\theta) + r^{2\lambda_{p} - 2m} f_{2} (\theta), \\
G^{2}(r,\ta) &= r^{2\lambda_{p} - 2} g_{1}(\theta) + r^{2\lambda_{p} - 2m} g_{2} (\theta), \\
H^{2} (r,\ta) &= r^{2\lambda_{p} - 2} h_{1}(\theta) + r^{2\lambda_{p} - 2m} h_{2} (\theta),
\end{split}
\end{equation}
where $f_{i},\, g_{i},\, h_{i}\in C^{\infty}(S^{1})$ ($\, i=1,2$).
For such expressions of $F^2,\, G^2$ and $H^2$, equation \eqref{Patton} becomes
\begin{equation} \label{The Shawshank Redemption}
\ds \frac{1}{\rho}  \begin{pmatrix}
 \varphi^{2}  \\
 \psi^{2}
 \end{pmatrix}_{\theta}  = \Lambda(\ta)
\begin{pmatrix}
 \varphi^{2}  \\
 \psi^{2}
 \end{pmatrix}_{\rho} +
\rho^{2\lambda_{p} - 2}V_1(\ta) +  \rho^{2\lambda_{p} - 2m} V_2(\ta),
\end{equation}
with $V_1,\, V_2\, \in C^\infty(S^1,\R^2)$.
We seek solutions of
\eqref{The Shawshank Redemption} in the form
\begin{equation} \label{The Bridge on the River Kwai}
\begin{pmatrix}
 \varphi^{2}  \\
 \psi^{2}
 \end{pmatrix}  = \rho^{2\lambda_{p} - 1} X_{1}(\theta) + \rho^{2\lambda_{p} - 2m + 1} X_{2}(\theta),
\end{equation}
with $X_1,\, X_2\,\in C^\infty(S^1,\R^2)$. This leads to the following equations for $X_1,\, X_2$:
\begin{equation} \label{Saving Private Ryan}
\begin{array}{ll}
\ds\frac{dX_{1}}{d\theta} = (2\lambda_{p} - 1) \Lambda(\ta) X_1+V_1(\ta)\, \ \ \ \ \ds\frac{dX_{2}}{d\theta} = (2\lambda_{p} - 2m + 1)\Lambda(\ta)X_2 +V_2(\ta)\, .
\end{array}
\end{equation}

It follows from the classical theory of differential equations with periodic coefficients
(see Section 2.9 of \cite{ys} for instance) that if $V_1$ and $V_2$ are not zero and if  both
$(2\lambda_{p} - 1)$ and $(2\lambda_{p} - 2m + 1)$ are not eigenvalues of the periodic system
\eqref{The Hurt Locker}, then \eqref{Saving Private Ryan} has periodic solutions $X_1$ and $X_2$
(note that if $V_i=0$, the corresponding equation has a trivial periodic solution).
The next step is to understand the asymptotic behavior of the spectrum of \eqref{The Hurt Locker}
and show that there exist arbitrarily large $p\in\N$ such that $(2\lambda_{p} - 1)$ and
$(2\lambda_{p} - 2m + 1)$ are indeed not eigenvalues of \eqref{The Hurt Locker}.

Let $\ds J = \begin{pmatrix}0 & 1 \\ -1 & 0\end{pmatrix}$ and write system \eqref{The Hurt Locker}
in the
standard Hamiltonian form
\begin{equation} \label{Goodfellas}
J X'(\theta) = \lambda H(\theta) X(\theta),
\end{equation}
with
\begin{equation} \label{Lawrence of Arabia}
H(\ta)=\begin{pmatrix}h_{11} & h_{12}\\ h_{21} &h_{22} \end{pmatrix}
=J\Lambda=\frac{1}{(m^2-m)r^{m-2}P(\ta)}
\begin{pmatrix}z_{tt} & -z_{st}\\ -z_{st} & z_{ss} \end{pmatrix}\, .
\end{equation}
The assumption on the curvature $K$ given by \eqref{Eyes Wide Shut} implies that
the matrix $H$ is positive except possibly
 at points $\ta_1,\, \cdots,\, \ta_\ell\in S^1$.
It follows then from the Oscillation Theorem (see Theorem V in \cite{ys}  pg. 766) that the spectrum
\eqref{Goodfellas} (or equivalently \eqref{The Hurt Locker}) consists of a sequence
\begin{equation} \label{The French Connection}
\lambda_{1}^{-} \leq \lambda_{1}^{+} < \lambda_{2}^{-} \leq \lambda_{2}^{+} <
\ldots\lambda_j^-\le  \lambda_{j}^{+} < \ldots,
\end{equation}
with $\ds \lim_{j\to\infty}\lambda_j^\pm=\infty$. To get the asymptotic behavior of $\lambda_j^\pm$ we
introduce the following functions:
\begin{equation} \label{Being John Malkovich}
b(\theta) = \sqrt{\textrm{det}(H(\ta))}=
\frac{1}{(m^2-m)r^{m-2}P(\ta)}\sqrt{z_{tt}z_{ss}-z_{st}^2}\, ,
\end{equation}
and
\begin{equation} \label{Psycho}\begin{array}{ll}
\ds c_1(\ta)& =\ds\frac{-h_{21}}{b}\,\frac{h'_{11}}{h_{11}}=
\frac{z_{st}}{\sqrt{z_{tt}z_{ss}-z_{st}^2}}\, \frac{d\log(z_{tt})}{d\ta}\\ \\
\ds c_2(\ta)& =\ds\frac{-h_{21}}{b}\,\frac{h'_{22}}{h_{22}}=
\frac{z_{st}}{\sqrt{z_{tt}z_{ss}-z_{st}^2}}\, \frac{d\log(z_{ss})}{d\ta}.
\end{array}\end{equation}

\begin{proposition} \label{Do the Right Thing}
Let $b$, $c_1$ and $c_2$ be as in \eqref{Being John Malkovich}, \eqref{Psycho} and
suppose that $c_1,\, c_2$ are elements of $\ L^1([0,2\pi])$. Let
\begin{equation} \label{Rain Man}
b_{1} = \int\limits_{0}^ {2\pi} b(\theta) d \theta\quad\textrm{and}\quad
b_{2} = - \ds \frac{1}{4}\int\limits_{0}^{2\pi} (c_1(\ta)-c_2(\ta))\,d\ta.
\end{equation}
The eigenvalues \eqref{The French Connection} of \eqref{The Hurt Locker} have the following
asymptotic behavior as $j\to \infty$:
\begin{equation} \label{Spotlight}
\lambda_{j}^{\pm}  = \ds \frac{j \pi}{b_{1}} + \ds \frac{b_{2}}{b_{1}} + O \left(\ds \frac{1}{j} \right). \end{equation}
\end{proposition}

\begin{proof} Note that $b(\ta)>0$ for $\ta\ne\ta_i$, $i=1,\,\cdots ,\,\ell$;
this implies  $b_{1}>0$.
For $\varepsilon > 0$, consider the Hamiltonian
 $H_{\varepsilon}(\theta) = H(\ta)+\varepsilon I= \begin{pmatrix}  h_{11}+ \varepsilon  & h_{12} \\
 h_{21} &  h_{22} + \varepsilon  \end{pmatrix}.$
 Since $z_{tt}$ and $z_{ss}$
are nonnegative, then $h_{11}+\varepsilon$ and $h_{22}+\varepsilon$ are strictly positive and
\[
\textrm{Det}(H_\varepsilon)=\textrm{Det}(H)+\varepsilon \textrm{Trace}(H)+\varepsilon^2 \, >0\, ,
\quad\forall\ta\in\R\, .
\]
Thus matrix $H_\varepsilon$ is positive for every $\varepsilon >0$ and
the asymptotic behavior of the spectrum of the equation
\begin{equation} \label{Rear Window}
JX_\varepsilon' = \lambda H_{\varepsilon}(\theta) X_\varepsilon
\end{equation}
is given by (see p.776 \cite{ys})
\begin{equation} \nonumber
\lambda_{\varepsilon j}^{\pm} = \ds \frac{j \pi }{b^{\varepsilon}_{1}} + \ds \frac{b_{2}^{\varepsilon}}{b_{1}^{\varepsilon}} + O \left(\ds \frac{1}{j} \right),
\end{equation}
where
\[
b_{1}^{\varepsilon} = \int\limits_{0}^ {2\pi} \sqrt{\textrm{Det}(H_\varepsilon)}\, d\ta\quad\textrm{and}\quad
b_{2}^{\varepsilon}= \frac{-1}{4}\int\limits_{0}^ {2\pi}\frac{h_{12}}{\textrm{Det}(H_\varepsilon)}
\left(\frac{h'_{11}}{h_{11}+\varepsilon}-\frac{h'_{22}}{h_{22}+\varepsilon}\right)\, d\ta\, .
\]
Moreover, the hypothesis $c_1,\, c_2\,\in L^1$ and the Dominated Convergence Theorem imply that
$\ds b_1=\lim_{\varepsilon\to 0^+}b_{1}^{\varepsilon}$ and
$\ds b_2=\lim_{\varepsilon\to 0^+}b_{2}^{\varepsilon}$.
Finally, since the monodromy matrix of \eqref{Rear Window} is analytic with respect to $\varepsilon$,
it follows that for $j$ large,
$\ds \lambda_j^\pm=\lim_{\varepsilon\to 0^+}\lambda_{\varepsilon j}^{\pm}$ and the asymptotic
expansion \eqref{Spotlight} follows.
\end{proof}

Now we can prove the main results of this section.
\begin{theorem}\label{Homogeneous1}
Let $S\subset\R^3$ be the graph of a homogeneous function $z(s,t)$ of order $m\ge 2$
and satisfying \eqref{Eyes Wide Shut}. Suppose that conditions in Proposition \ref{Do the Right Thing} hold  and $b_1,\ b_2$ given by \eqref{Rain Man} satisfy
\begin{equation} \label{Born on the Fourth of July}
b_1-b_2\,\notin\,\pi\Z\quad\textrm{and}\quad
(2m-1)b_1-b_2\,\notin\,\pi\Z\, .
\end{equation}
Then for every $k \in \N$, there exist  $U^{1},\, U^{2}\,\in C^k( \R^2 ,\R^{3})$
such that the deformation of $S$ given by the position vector
\[
R(s,t)+2\epsilon U^1(s,t)+\epsilon^2 U^2(s,t)
\]
is a nontrivial infinitesimal bending of order 2.
\end{theorem}

\begin{proof}
Let $k\in\N$; we already have the first field $U^1\in C^k$ given by \eqref{The Terminator} (with $p\in\N$ large enough).
Now we construct
$U^2=(u^2,v^2,w^2)$ through the functions $\varphi^2$, $\psi^2$ given by \eqref{The Bridge on the River Kwai}
and satisfying
\eqref{The Shawshank Redemption}. Let
\begin{equation*} \nonumber
\delta_{1} = \ds\min \left\{ \left| \ds \frac{a \pi}{b_{1}} + \ds \frac{b_{2}}{b_{1}} - 1 \right|, \ a \in \Z  \right\}, \ \ \delta_{2}  = \ds\min \left\{\left| \ds \frac{a \pi}{b_{1}} + \ds \frac{b_{2}}{b_{1}} +1 - 2m \right|, \ a \in \Z  \right\}\, 
\end{equation*}
and $\delta  = \ds\min \left\{\delta_{1}, \delta_{2} \right\}$. It follows from \eqref{Born on the Fourth of July} that $\delta > 0$.
Thanks to the asymptotic expansion \eqref{Spotlight}, we can choose $p$ large enough so that
\begin{equation}\label{rp}
\lambda_{p}^{\pm}  = \ds \frac{p \pi}{b_{1}} + \ds \frac{b_{2}}{b_{1}} + r_{p},
\end{equation}
with $|r_{p}| < \ds \frac{\delta}{6}$.
Now we show that there is no $q \in \N$ such that
$$2 \lambda_{p}^{\pm} - 1 = \lambda_{q}^{\pm}.$$
It follows from \eqref{Spotlight} that
\begin{equation} \nonumber
2 \lambda_{p}^{\pm} - 1 - \lambda_{q}^{\pm} = (2p - q) \ds \frac{\pi}{b_{1}} + \ds \frac {b_{2}}{b_{1}} - 1 + 2r_{p} - r_{q}.
\end{equation}

We can assume with no loss of generality that $q > p$ and thus $|r_{q}| < \ds \frac{\delta}{6}$. This implies that
\begin{equation} \nonumber
\left|2 \lambda_{p}^{\pm} - 1 - \lambda_{q}^{\pm} \right| \geq \left|(2p - q) \ds \frac{\pi}{b_{1}} + \ds \frac {b_{2}}{b_{1}} - 1 \right| - \left|2r_{p} - r_{q} \right| \geq \ds \frac{\delta}{2}>0\, .
\end{equation}
A similar argument shows that there is no $q \in \N$ such that
$$2 \lambda_{p}^{\pm} - 2m + 1 = \lambda_{q}^{\pm}.$$
As a consequence,
the periodic system \eqref{Saving Private Ryan} has solutions $X_1(\ta)$, $X_2(\ta)$.
We have therefore $\ds\begin{pmatrix} \varphi^2\\ \psi^2\end{pmatrix}$
a solution of \eqref{The Shawshank Redemption} in the form \eqref{The Bridge on the River Kwai}.
Using \eqref{Star Wars}, \eqref{The Empire Strikes Back}, and \eqref{Return of Jedi}, we get
\begin{equation} \label{Chinatown}
w^{2} = \ds\frac{\varphi^{2}_{s} + \psi^{2}_{t} - F^{2} - H^{2}}{z_{ss} + z_{tt}}
= r^{2 \lambda_{p} - m} \gamma_{1}(\theta) +  r^{2 \lambda_{p} - 3m + 2} \gamma_{2}(\theta),
\end{equation}
and then from   \eqref{Citizen Kane} and \eqref{Chinatown} we obtain
\begin{equation} \label{Jaws}
\begin{split}
u^{2} &=  r^{2 \lambda_{p} - 1} \alpha_{1}(\theta) + r^{2 \lambda_{p} - 2m + 1} \alpha_{2}(\theta), \\
v^{2} &=  r^{2 \lambda_{p} - 1} \beta_{1}(\theta) + r^{2 \lambda_{p} - 2m + 1} \beta_{2}(\theta), \\
\end{split}
\end{equation}
with $\alpha_i,\, \beta_i,\, \gamma_i\, \in C^\infty(S^1)$ for $i=1,2$. By taking $p\in\N$ large enough,
we get $U^1,\, U^2\, \in C^k(\R^2,\R^3)$.
\end{proof}

\begin{theorem}\label{Homogeneous2}
Let $S\subset\R^3$ be the graph of a homogeneous function $z(s,t)$ of order $m\ge 2$
and satisfying \eqref{Eyes Wide Shut}. Suppose that conditions in Proposition \ref{Do the Right Thing} hold  and $b_1,\ b_2$ given by \eqref{Rain Man} satisfy
\begin{equation} \label{Born on the Fourth of July 2}
\left(b_2\mathbb{N}+b_1\mathbb{Z}\right)\,\cap\,
\left(mb_1\mathbb{N}+\pi\mathbb{Z}\right)\, =\, \emptyset.
\end{equation}
Then for every $k, l \in \N$, there exist
 $U^{1},\,\cdots\, , U^{l}\,\in  C^k( \R^2 ,\R^{3})$
such that the deformation of $S$ given by the position vector
\[
R(s,t)+2\epsilon U^1(s,t)+\cdots 2\epsilon^l U^l(s,t)
\]
is a nontrivial infinitesimal bending of order $l$.
\end{theorem}

\begin{proof}
We use an induction argument on the order $l$. Suppose that
there exist $U^{1},\,\cdots\, , U^{l-1}\,\in  C^k( \R^2 ,\R^{3})$
such that
\[
R(s,t)+2\epsilon U^1(s,t)+\cdots 2\epsilon^{l-1} U^{l-1}(s,t)
\]
is a nontrivial infinitesimal bending of order $l-1$ and
for $q\in\{ 1,\cdots , l-1\}$,
$U^q(r,\ta )=\left(u^q(r,\ta),v^q(r,\ta), w^q(r,\ta) \right)$, with
\begin{equation}\label{Hollywood}\begin{array}{ll}
u^q(r,\ta) & = r^{q\lambda_p-m\nu_1+\mu_1}\alpha_1(\ta)+\cdots+
r^{q\lambda_p-m\nu_\sigma+\mu_\sigma}\alpha_\sigma(\ta),\\
v^q(r,\ta) & = r^{q\lambda_p-m\nu_1+\mu_1}\beta_1(\ta)+\cdots+
r^{q\lambda_p-m\nu_\sigma+\mu_\sigma}\beta_\sigma(\ta),\\
w^q(r,\ta) & = r^{q\lambda_p-m(\nu_1+1)+\mu_1+1}\gamma_1(\ta)+\cdots+
r^{q\lambda_p-m(\nu_\sigma+1)+\mu_\sigma+1}\gamma_\sigma(\ta)\, ,
\end{array}
\end{equation}
where $\sigma,\, \nu_1,\, \mu_1,\, \cdots\, ,\nu_\sigma,\, \mu_\sigma$ are
integers that depend only on the index $q$,\\
$\, \alpha_i,\, \beta_i,\, \gamma_i\,\in C^\infty(S^1)$ for $i=1,\cdots ,\sigma$, and where
$\lambda_p$ is a spectral value of the equation \eqref{Indiana Jones and the Last Crusade}
that can be chosen arbitrarily large.
Note that the components of the bending fields $U^1$ and $U^2$ constructed in the proof of
Theorem \ref{Homogeneous1} have the form given in \eqref{Hollywood}
(see \eqref{The Terminator}, \eqref{Chinatown}, and \eqref{Jaws}).

It follows from \eqref{Hollywood}   that the functions
$F^l$, $G^l$, and $H^l$ defined by \eqref{All the President's Men} have the form
\begin{equation}\label{Waco}\begin{array}{ll}
F^l(r,\ta) & = r^{l\lambda_p-m s_1+t_1}f_1(\ta)+\cdots+
r^{l\lambda_p-ms_\tau+t_\tau}f_\tau(\ta),\\
G^l(r,\ta) & = r^{l\lambda_p-m s_1+t_1}g_1(\ta)+\cdots+
r^{l\lambda_p-ms_\tau+t_\tau}g_\tau(\ta),\\
H^l(r,\ta) & = r^{l\lambda_p-m s_1+t_1}h_1(\ta)+\cdots+
r^{l\lambda_p-ms_\tau+t_\tau}h_\tau(\ta),\\
\end{array}
\end{equation}
where $\tau,\, s_1,\, t_1,\, \cdots\, ,s_\tau,\, t_\tau$ are
integers that depend only on the index $l$, and where
$\, f_i,\, g_i,\, h_i\,\in C^\infty(S^1)$ for $i=1,\cdots ,\tau$.
With $F^l$, $G^l$, and $H^l$ as in \eqref{Waco}, we construct the next
bending field $U^l=(u^l,v^l,w^l)$ through the functions $\varphi^l$ and $\psi^l$
given by \eqref{Citizen Kane}. In this situation equation, \eqref{Patton} becomes
\begin{equation} \label{Django}
\ds \frac{1}{\rho} \begin{pmatrix}
 \varphi^{l}  \\
 \psi^{l}
 \end{pmatrix}_{\theta}  = \Lambda(\ta)
\begin{pmatrix}
 \varphi^{l}  \\ \psi^{l}\end{pmatrix}_{\rho} +
 \rho^{l\lambda_p-m s_1+t_1}V_1(\ta)+\cdots +
 \rho^{l\lambda_p-m s_\tau+t_\tau}V_\tau(\ta)\, ,
\end{equation}
with $V_1,\cdots ,V_\tau \in C^\infty(S^1,\R^2)$.  We seek solutions of
\eqref{Django} in the form
\begin{equation}\label{Benjamin Button}
\begin{pmatrix}
 \varphi^{l}  \\
 \psi^{l}
 \end{pmatrix} =\rho^{l\lambda_p-m s_1+t_1+1}X_1(\ta)+\cdots +
 \rho^{l\lambda_p-m s_\tau+t_\tau +1}X_\tau(\ta),
\end{equation}
with $X_1,\cdots ,X_\tau \in C^\infty(S^1,\R^2)$.
This leads to the following periodic differential equations for the $X_i$'s:
\begin{equation}\label{Minority Report}
X_i'(\ta)=\left(l\lambda_p-m s_i+t_i+1\right)\,\Lambda(\ta)\, X_i(\ta) +V_i(\ta)\, .
\end{equation}
By using the asymptotic expansion \eqref{rp}  of the spectral value $\lambda_p$
(with $r_p$ arbitrarily small for $p$ large enough),
an argument similar to that used in the proof of Theorem \ref{Homogeneous1} shows that for
$p$ large enough, $\, l\lambda_p-ms_\sigma+t_\sigma +1$ cannot be a spectral value.
Indeed, if there were arbitrarily large $p\in\mathbb{N}$ such that
\[
l\lambda_p-ms_\sigma+t_\sigma +1 =\lambda_q\, ,
\]
then necessarily
\[
(l-1)b_2+(t_\sigma +1)b_1=m b_1 s_\sigma +(q-lp)\pi\ + r_{p,q} ,
\]
with $r_{p,q} \to 0$ when $p,q \to \infty$, which contradicts hypothesis \eqref{Born on the Fourth of July 2}.
This implies that the system of equations \eqref{Minority Report} has a periodic solution and so
there exist $\varphi^l$, $\psi^l$ as in \eqref{Benjamin Button} satisfying \eqref{Django}.
Since $w^l=\ds \frac{\varphi^l_s+\psi^l_t-F^l-H^l}{z_{ss}+z_{tt}}\,$,
then it follows from \eqref{Waco} and \eqref{Benjamin Button} that
\[\left\{\begin{array}{ll}
w^l & \ds =\sum_{j=1}^lr^{l\lambda_p-m(s_j+1)+t_j+2}\gamma_j(\ta)\, ,\\
u^l & \ds =\varphi^l-z_sw^l=\sum_{j=1}^lr^{l\lambda_p-ms_j+t_j+1}\alpha_j(\ta)\, ,\\
v^l & \ds =\psi^l-z_tw^l=\sum_{j=1}^lr^{l\lambda_p-ms_j+t_j+1}\beta_j(\ta)\, .\\
\end{array}\right.\]
Thus for a given $k\in\mathbb{N}$,  if $p$ is large enough, $U^l=(u^l,v^l,w^l)$ as in \eqref{Hollywood}
is in $C^k(\R^2,\R^3)$
and $\ds R(s,t)+2\sum_{j=1}^l\epsilon^jU^j(s,t)$ is an infinitesimal bending of order $l$.
\end{proof}

\begin{remark}
When $m\in\mathbb{N}$, condition \eqref{Born on the Fourth of July 2} reduces to $\left( b_2\mathbb{N}+b_1\Z\right)\cap\pi \Z =\emptyset\, .$
\end{remark}

\section{Analytic Infinitesimal Bendings of a  Class of Surfaces}
We describe here the structure of all analytic infinitesimal bendings of a
particular class of surfaces given as the graph of functions of the form
$s^{m+2}\pm t^{n+2}$, with $m,n$ positive integers. We show that the space
of such infinitesimal bendings is generated by four arbitrary analytic functions of
one real variable. The basic ingredient needed 
is the equation in $\R^2$ given by
\begin{equation} \label{the good}
x^mw_{yy}\, +\varepsilon y^nw_{xx}=0,\quad \varepsilon =1\ \textrm{or}\ -1\, .
\end{equation}
To describe the analytic solutions of \eqref{the good}
we will need some technical lemmas.

\begin{lemma}\label{doubleseries1}
The double series
$\ds \sum_{m,k\ge 0}\left(\frac{(m+k)!}{m!\, k!}\right)^2 X^kY^m $
is uniformly convergent in the bidisc $|X|<1/4$ and $|Y|<1/4$.
\end{lemma}

\begin{proof}
Note that $\ds \frac{(m+k)!}{m!\, k!}\le 2^{m+k}$. Therefore
\[
\sum_{m,k\ge 0}\left(\frac{(m+k)!}{m!\, k!}\right)^2 |X|^k|Y|^m\le
\sum_{m,k\ge 0} (4|X|)^m(4|Y|)^k
\]
and the conclusion follows.
\end{proof}

We will use the following notations: $D$ stands for the operator
$\ds D=\frac{1}{z^m}\,\frac{d^2}{dz^2}$ where $z$ denotes a complex or real variable; and
for $\alpha,\,  \beta\in\Z^+$,
\[ A_\beta^\alpha=\prod_{k=1}^\beta[k(\alpha+2)-1][k(\alpha+2)]\quad\textrm{and}\quad
 B_\beta^\alpha=\prod_{k=1}^\beta[k(\alpha+2)][k(\alpha+2)+1]\,. \]
We will also denote $A_0^\alpha =1$ and $B_0^\alpha =1$.

\begin{lemma}\label{convergence1}
Let $h(z)$ be a holomorphic function in the disc $D(0,R)$ and let $M(X,Y)$ be the function
defined by
\[
M(X,Y)=\sum_{j\ge 0}H_j(X^{m+2})Y^{j(n+2)}\, ,
\]
where
\[
H_j(X^{m+2})=\frac{1}{A_j^n}D^j\left[h(X^{m+2})\right]\, .
\]
There exists a positive constant $C=C(m,n)$ that depends only on $m,n$ such that
the function $M$ is holomorphic for $\ds |X| <CR^{1/(m+2)}$ and $\ds |Y| <CR^{1/(n+2)}$.
In particular, if $h$ is an entire function in $\C$, then $M$ is an entire function in $\C^2$.
\end{lemma}

\begin{proof}
First note that for $q\ge 0$, $\ds D(X^{q(m+2)})=[q(m+2)]\,[q(m+2)-1]X^{(q-1)(m+2)}$
and it is easily verified  that for $j\le q$
\[
D^j\left(X^{q(m+2)}\right)=\frac{A^m_q}{A^m_{q-j}}X^{(q-j)(m+2)}\, .
\]
Let $0<\rho <R$ be arbitrary. For $j\ge 0$ and $|X|^{m+2}<\rho$ we have
\begin{align*}\begin{array}{ll}
D^j\left[h(X^{m+2})\right] &=\ds\frac{1}{2\pi i}
\int_{|\zeta|=\rho}\!\! h(\zeta)D^j\left(\frac{1}{\zeta -X^{m+2}}\right)\, d\zeta\\ \\
&
\ds =\frac{1}{2\pi i}\sum_{q\ge j}\int_{|\zeta|=\rho}\!\! \frac{h(\zeta)}{\zeta^{1+q}}D^j(X^{q(m+2)})\, d\zeta\\ \\
& \ds =\frac{1}{2\pi i}\sum_{p\ge 0}\frac{A^m_{p+j}}{A^m_p}\int_{|\zeta|=\rho}\!\!
\frac{h(\zeta)}{\zeta^{1+j}}\left(\frac{X^{m+2}}{\zeta}\right)^p\, d\zeta\, .
\end{array}\end{align*}
Thus
\[
\left|D^j\left[h(X^{m+2})\right]\right|\le\, ||h||\,\sum_{p\ge 0}\frac{A^m_{p+j}}{A^m_p}\frac{1}{\rho^j}
\left(\frac{|X|^{m+2}}{\rho}\right)^p,
\]
where $||h||$ denotes the maximum of $h$ in the disc. It follows that
\begin{align}\label{M-estimate}\begin{array}{ll}
|M(X,Y)| & \le \ds \sum_{j\ge 0} |H_j(X^{m+2})| \, |Y^{n+2}|^j \\
&\le \ds ||h|| \, \sum_{j\ge 0 }\sum_{p\ge 0}\frac{A^m_{p+j}}{A^m_pA^n_j}\left(\frac{|X|^{m+2}}{\rho}\right)^p
\left(\frac{|Y|^{n+2}}{\rho}\right)^j\, .
\end{array}\end{align}
Now we estimate the coefficient $\ds \frac{A^m_{p+j}}{A^m_pA^n_j}$. For $\alpha ,\, \beta\in\Z^+$, we have
\begin{align*}\begin{array}{ll}
A_\beta^\alpha & =\ds (\alpha +2)^{2\beta}\prod_{k=1}^\beta k^2\left(1-\frac{1}{k(\alpha +2)}\right) \, \le (\alpha+2)^{2\beta} (\beta!)^2,\\
A_\beta^\alpha & \ge \ds (\alpha +2)^{2\beta}\prod_{k=1}^\beta k^2\left(1-\frac{1}{(\alpha +2)}\right) =
(\alpha+2)^{\beta}(\alpha+1)^\beta (\beta!)^2\, .
\end{array}\end{align*}
These inequalities imply that
\begin{equation}\label{combinatorial}
\frac{A^m_{p+j}}{A^m_pA^n_j} \le
\left(\frac{m+2}{m+1}\right)^p\, \left(\frac{(m+2)^2}{(n+2)(n+1)}\right)^j\left(\frac{(p+j)!}{p!\,j!}\right)^2\, .
\end{equation}
It follows from estimates \eqref{M-estimate}  and \eqref{combinatorial}  that
\begin{equation}\label{M-estimate2}
|M(X,Y)|\le ||h||\, \sum_{j,p}\!\frac{((p+j)!)^2}{(p!)^2\,(j!)^2}
\left[\frac{m+2}{m+1}\frac{|X|^{m+2}}{\rho}\right]^p\,
\left[\frac{(m+2)^2}{(n+2)(n+1)}\frac{|Y|^{n+2}}{\rho}\right]^j.
\end{equation}
Finally the conclusion follows from \eqref{M-estimate2} and Lemma \ref{doubleseries1} where  the constant $C$ can be
taken as
\begin{equation} \label{Seven}
C=\min\left[ \left(\frac{m+1}{4(m+2)}\right)^{1/(m+2)},\left(\frac{(n+2)(n+1)}{4(m+2)^2}\right)^{1/(n+2)} \right]\, .
\end{equation}
\end{proof}

 Then we have the following proposition
\begin{proposition}\label{w-equation}
Let $R > 0$ and  $C$ given in \eqref{Seven}.
A function $w$ is an analytic solution of \eqref{the good} for $|x|, |y|<R$ if
and only if there exist $h^1,\, h^2,\, h^3\,$ and $h^4$ analytic functions of one real variable $t$ in
the interval $|t|< \min \left\{\left(\ds \frac{R}{C} \right)^{m+2}, \left(\ds \frac{R}{C} \right)^{n+2} \right\}$
such that
\begin{equation}\label{the bad}
w(x,y)=\sum_{p=0}^\infty
\left[H^1_{p}(x^{m+2})+xH^2_{p}(x^{m+2})+yH^3_{p}(x^{m+2})+
xyH^4_{p}(x^{m+2})\right]y^{p(n+2)}\, ,
\end{equation}
where
\begin{equation}\label{the ugly}
\begin{array}{lll}
\ds H^i_p(x^{m+2}) & =\ds \frac{(-\varepsilon)^p}{A^n_p}\, D^p\left(h^i(x^{m+2})\right) &\quad\textrm{for}\ i=1,\ 2\, ;\\ \\
\ds H^i_p(x^{m+2}) & =\ds \frac{(-\varepsilon)^p}{B^n_p}\, D^p\left(h^i(x^{m+2})\right) &\quad\textrm{for}\ i=3,\ 4\, ,
\end{array}\end{equation}
 In particular, $w$ is analytic on $\R^{2}$ if and only if $h_{1}, h_{2}, h_{3}$ and $h_{4}$ are analytic on $\R$.
\end{proposition}

\begin{proof}
Suppose that $w(x,y)$ is an analytic solution of \eqref{the good}. 
We expand $w$ with respect to $y$ as:
\begin{equation}\label{ozark}
w(x,y)=\sum_{j=0}^\infty\alpha_j(x)y^j\, ,
\end{equation}
where the $\alpha_j$ are real analytic functions of $x$. Equation \eqref{the good}
leads to
\[
\sum_{j=0}^{n-1}(j+2)(j+1)x^m\alpha_{j+2}(x)y^j+
\sum_{j=n}^{\infty}\left[(j+2)(j+1)x^m\alpha_{j+2}(x)+\varepsilon\alpha''_{j-n}(x)\right]y^j=0.
\]
Therefore,
\begin{equation}\label{chaplin}\left\{\begin{array}{ll}
\alpha_k(x) & =0\quad \textrm{for}\ k=2,\,\cdots\, ,n+1\, , \\
\alpha_k(x) & =\ds \frac{-\varepsilon}{k(k-1)}D\alpha_{k-(n+2)} \quad \textrm{for}\ k\ge n+2\, ,
\end{array}\right.\end{equation}
where $D$ is the operator defined above. It follows at once by induction  from \eqref{chaplin} that
$\alpha_k =0$ whenever $k\ \textrm{ and} \ k-1\notin (n+2)\Z^+$.

For $k=n+2$, we have $\alpha_{n+2}=\ds\frac{-\varepsilon}{A_1}D\alpha_0$ and then for $k=2(n+2)$
we get
\[
\alpha_{2(n+2)}=\ds\frac{-\varepsilon}{[2(n+2)][2(n+2)-1]}D\alpha_{n+2}=\frac{(-\varepsilon)^2}{A_2}D^2\alpha_0\, .
\]
An induction shows that
\begin{equation}\label{Orson}
\alpha_{p(n+2)}=\ds\frac{(-\varepsilon)^p}{A^n_p}D^p\alpha_0\,, \ \ \ \alpha_{p(n+2)+1}=\ds\frac{(-\varepsilon)^p}{B^n_p}D^p\alpha_1\,.
\end{equation}
 
Since each $\alpha_{p(n+2)}$ and $\alpha_{p(n+2)+1}$ is an analytic function of $x$, then \eqref{Orson}  imposes restrictions on the power series representations
$\alpha_0(x)=\ds \sum a_jx^j$ and $\alpha_1(x)=\ds\sum b_jx^j$. Indeed, we have
\[
\alpha_{n+2}(x)=\frac{-\varepsilon}{A^n_1}D\alpha_0(x)=\frac{-\varepsilon}{A^n_1}\sum_{j=2}^\infty j(j-1)a_jx^{j-(m+2)},
\]
and so $a_2=\,\cdots\, =a_{m+1}=0$. This gives
\[
a_{n+2}(x)=\frac{-\varepsilon}{A^n_1}\sum_{k=m+2}^\infty \left[k+(m+2)\right]\left[k+(m+1)\right]a_{k+(m+2)}x^{k}\, .
\]
By repeating the above argument and an induction on the relation
$\alpha_{(p+1)(n+2)}=M D\alpha_{p(n+2)}$ (with $M$ nonzero constant), we can prove that
\[
a_j=0\ \textrm{if}\ j\ne q(m+2)\ \textrm{and}\ j\ne q(m+2)+1\ \textrm{with}\ q\in\Z\, .
\]
We have then
\[
\alpha_0(x)=\sum_{r=0}^\infty a_{r(m+2)}x^{r(m+2)}+\sum_{r=0}^\infty  a_{r(m+2)+1}x^{r(m+2)+1}\, .
\]
Let $h^1(t)=\ds\sum_{r=0}^\infty a_{r(m+2)}t^r$ and $h^2(t)=\ds\sum_{r=0}^\infty a_{r(m+2)+1}t^r$.
Then
\begin{equation}\label{John}
\alpha_0(x)=h^1(x^{m+2}) +xh^2(x^{m+2})\, .
\end{equation}

Similar arguments show that there exist analytic functions $h^3(t)$ and $h^4(t)$ such that
\begin{equation}\label{Wayne}
\alpha_1(x)=h^3(x^{m+2}) +xh^4(x^{m+2})\, .
\end{equation}
Expression \eqref{the bad} of the proposition follows from \eqref{ozark}, \eqref{chaplin}, \eqref{Orson}, \eqref{John}, and \eqref{Wayne} and the convergence follows from Lemma \ref{convergence1}.
Conversely, given analytic functions $h^1,\, h^2,\, h^3\,$ and $h^4$ as in the statement of the proposition,
it is clear that the function $w(x,y)$ defined by \eqref{the bad} is an analytic solution of the \eqref{the good}
\end{proof}

For $m,n\in \Z^+$, let $S_{m,n}\subset\R^3$ be the surface given by
\[
S_{m,n}=\left\{ (s,t, s^{m+2}+\varepsilon t^{n+2}),\ \ (s,t)\in\R^2 \right\}\, .
\]
Denote by $\mathrm{IB}_{m,n}(\rho)$ the space of real analytic infinitesimal bendings of $S_{m,n}$
in the square $|s|<\rho$, $|t|<\rho$, with $0<\rho \le \infty$, and by
$\mathcal{A}(\rho)$
the space of $\R$-valued real analytic functions in the interval $(-\rho\, ,\ \rho)$.
Then we have the following theorem.

\begin{theorem}\label{AnalyticBendings}
$\mathrm{IB}_{m,n}(\rho)$
is isomorphic to $\mathcal{A}(\rho)^4$.
\end{theorem}

\begin{proof}
Let $U(s,t)=\left(u(s,t),v(s,t),w(s,t)\right)\in \mathrm{IB}_{m,n}(\rho)$. In this particular
case the system of equations \eqref{Dr. Strangelove} becomes
\begin{align*}
u_s+(m+2)s^{m+1}w_s & =0\\
u_t+v_s+(m+2)s^{m+1}w_t+\varepsilon (n+2)t^{n+1}w_s & =0\\
v_t +\varepsilon (n+2)t^{n+1}w_t & =0
\end{align*}
As in \cite{v} we can reduce this system into a single equation for $w$ to obtain
\begin{equation}\label{clint}
(m+2)(m+1)s^mw_{tt}+\varepsilon (n+2)(n+1)t^nw_{ss}=0\, ,
\end{equation}
and any solution of \eqref{clint} gives rise to an infinitesimal bending $U$.
Equation \eqref{clint} is equivalent to \eqref{the good} through the linear change
of variables $x=Ps$ and $y=Qt$ where
\begin{align*}
\ds P=\left[(m+2)(m+1)\right]^{\frac{n}{mn-4}}\left[(n+2)(n+1)\right]^{\frac{2}{mn-4}},\\
\ds Q=\left[(n+2)(n+1)\right]^{\frac{m}{mn-4}}\left[(m+2)(m+1)\right]^{\frac{2}{mn-4}}\, ,
\end{align*}
when $mn\ne 4$. When $mn=4$, another simple linear change of variables transforms equation \eqref{clint}
to \eqref{the good}.
Proposition \ref{w-equation} establish
an isomorphism between real analytic solutions of \eqref{clint} and $\mathcal{A}(C\rho^\mu)^4$
(for some $C$ and $\mu$ positive) and
thus between $\mathrm{IB}_{m,n}(\rho)$ and
$\mathcal{A}(C\rho^\mu)^4$ . Finally, since $\mathcal{A}(C\rho^\mu)$ is
clearly isomorphic
to $\mathcal{A}(\rho)$, this completes the proof.
\end{proof}

\begin{remark}
It should be noted that surfaces given as graphs of functions of the form $f(s)+g(t)$ are bendable
and some bendings are given in \cite{d}. However, the above theorem characterizes all real analytic infinitesimal
bendings.
\end{remark}

\vskip 0.1in

\noindent{\bf Acknowledgements.} Part of this work was done when the first author was visiting the Department of Mathematics \& Statistics at FIU (Florida International
University). He would like to thank the members of the institution for the support provided during his visit.


\begin{thebibliography}{ABCD}

\bibitem[AU]{au} Achil'diev, A. I., \& Usmanov, Z. D., {\em Rigidity of a surface with a point of flattening}, Matematicheskii Sbornik, \textbf{115(1)}(1967), 89-96.



\bibitem[D]{d} Dorfman, A. G., {\em Solution of the Bending Equation for Certain Classes of Surfaces}, Uspekhi Matematicheskikh Nauk, \textbf{12(2)}(1957), 147-150.

\bibitem[LVM]{lvm} de Lessa Victor, B.,  \&  Meziani, A., {\em  A Generalized CR equation with isolated singularities}, 
preprint (2021).


\bibitem[M1]{m1} Meziani, A., {\em Nonrigidity of a class of two dimensional surfaces with positive curvature and planar points.}, Proceedings of the American Mathematical Society, {\bf 141.6} (2013), 2137-2143.

\bibitem[M2]{m2} Meziani, A.,  {\em Infinitesimal Bendings of Surfaces With Nonnegative Curvature}, Recent Progress on Some Problems in Several Complex Variables and Partial Differential Equations: International Conference, Partial Differential Equations and Several Complex Variables, Wuhan University, Wuhan, China, June 9-13, 2004 [and] International Conference, Complex Geometry and Related Fields, East China Normal University, Shanghai, China, June 2-24, 2004, {\bf  Vol. 400}, American Mathematical Soc., 2006.

\bibitem[M3]{m3} Meziani, A., {\em Infinitesimal bendings of homogeneous surfaces with nonnegative curvature}, Communications in Analysis and Geometry, \textbf{11(4)} (2003), 697-719.

\bibitem[M4]{m4} Meziani, A., {\em Solvability of planar complex vector fields with applications to deformation of surfaces}, Complex Analysis, Birkhäuser Basel (2010), 263-278.

\bibitem[N]{ni}
Niordson, F. I. {\em Shell Theory}, North-Holland Series Appl. Math. and Mechanics, Vol. 29, Amesterdam, (1985).

\bibitem[P]{po}
Pogorelov, A. V. {\em Bendings of surfaces and stability of shells}, Nauka, Moscow: Engl. Transl. AMS, Providence, R.I. (1988).

\bibitem[R]{ro} Rozendorn, E. R. {\em Surfaces of negative curvature}, Geometry III, Springer, Berlin, Heidelberg, (1992), 87-178.

\bibitem[S]{s} Sabitov, I. K. {\em Local theory of bendings of surfaces}, Geometry III, Springer, Berlin, Heidelberg, (1992), 179-250.

\bibitem[U]{u} Usmanov, Z. D., {\em On infinitesimal deformations of surfaces of positive curvature with an isolated flat point}, Mathematics of the USSR-Sbornik, \textbf{12(4)} (1970), 595-614.

\bibitem[V]{v} Vekua, I.N., {\em Generalized Analytic Functions}, Pergamon Press, (1962).

\bibitem[YS]{ys} Iakubovich, V. A., \& Starzhinskii, V. M., {\em Linear Differential Equations with Periodic Coefficients}, Wiley, (1975).


\end{thebibliography}
\end{document}